\documentclass[11pt, reqno, oneside, notitlepage]{article}

\usepackage[a4paper, total={6in, 9.1in}]{geometry}

\usepackage{amsmath,amscd}
\usepackage{amssymb}
\usepackage{amsthm}
\usepackage{comment}
\usepackage{graphicx}
\usepackage{epstopdf} 
\usepackage{mathrsfs}
\usepackage{authblk}

\usepackage{bm}
\usepackage{hyperref,cite}
\usepackage{xcolor}
\hypersetup{
	colorlinks,
	linkcolor={blue},
	urlcolor={blue},
	citecolor={red}
}

\newtheorem{thm}{Theorem}[section]

\newtheorem{lem}[thm]{Lemma}
\newtheorem{cor}[thm]{Corollary}

\newtheorem{defi}[thm]{Definition}

\newtheorem*{proposition*}{Proposition}

\numberwithin{equation}{section}

\newcommand{\norm}[1]{\lVert #1 \rVert}

\title{Determining a parabolic-elliptic-elliptic system by boundary observation of its non-negative solutions under chemotaxis background}

\author[1,*]{Yuhan Li}
\author[1,$\dagger$]{Hongyu Liu}
\author[2,$\natural$]{Catharine W. K. Lo}
\affil[1]{Department of Mathematics, City University of Hong Kong, Hong Kong, P.R. China}
\affil[2]{School of Mathematical Sciences, Shenzhen University, Shenzhen, 518061, P.R. China}
\affil[*]{yuhli2-c@my.cityu.edu.hk}
\affil[$\dagger$]{hongyu.liuip@gmail.com, hongyliu@cityu.edu.hk}
\affil[$\natural$]{catharinelowk@gmail.com, cwklo@szu.edu.cn}
\date{}

\begin{document}
\maketitle
\begin{abstract}
This paper addresses a profoundly challenging inverse problem that has remained largely unexplored due to its mathematical complexity: the unique identification of all unknown coefficients in a coupled nonlinear system of mixed parabolic-elliptic-elliptic type using only boundary measurements. The system models attraction-repulsion chemotaxis—an advanced mathematical biology framework for studying sophisticated cellular processes—yet despite its significant practical importance, the corresponding inverse problem has never been investigated, representing a true frontier in the field. The mixed-type nature of this system introduces significant theoretical difficulties that render conventional methodologies inadequate, demanding fundamental extensions beyond existing techniques developed for simpler, purely parabolic models. Technically, the problem presents formidable obstacles: the coupling between parabolic and elliptic components creates inherent analytical complications, while the nonlinear structure resists standard approaches. From an applied perspective, the biological relevance adds another layer of complexity, as solutions must maintain physical interpretability through non-negativity constraints. Our work provides a complete theoretical framework for this challenging problem, establishing rigorous unique identifiability results that create a one-to-one correspondence between boundary data and the model's parameters. We demonstrate the power of our general theory through a central biological application: the full parameter recovery for an attraction-repulsion chemotaxis model with logistic growth, thus opening new avenues for quantitative analysis in mathematical biology.
~\\\\
\textbf{Keywords:} Nonlinear parabolic-elliptic-elliptic system; chemotaxis; mixed-type equations; unique identifiability; simultaneous recovery; multiplicative separable form.~\\
\textbf{2020 Mathematics Subject Classification:} 35R30, 92-10, 35Q92, 35B09, 35K99, 35J99
\end{abstract}

{\centering
\section{INTRODUCTION}\label{sec:pee_Intro}}

\subsection{Biology background and motivation of our study}\label{sec:pee_BioBack}
Nonlinear partial differential equation (PDE) systems play a central role in the mathematical modeling of biological processes involving spatial organization and collective behavior. A particularly influential example arises in the study of chemotaxis—the directed movement of cells or organisms in response to chemical gradients. Such systems are fundamental for understanding a wide range of biological phenomena, including bacterial aggregation, immune response, embryonic development, and tumor growth. This section introduces a class of chemotaxis models that form the basis of our study, beginning with classical formulations and progressively extending to more complex frameworks that incorporate multiple chemical signals and growth dynamics.

The mathematical foundation of chemotaxis modeling was established in the 1970s through the pioneering work of Keller and Segel \cite{Keller1970initiation, Keller1971model, Keller1971traveling}. Their models were derived from empirical observations of bacterial movement in response to chemical attractants, capturing the tendency of cells to aggregate and form spatial patterns. The classical Keller-Segel model consists of a pair of reaction-diffusion equations that describe the interaction between cell density and chemoattractant concentration. It incorporates both random diffusion and advective transport due to chemotactic drift:
\begin{equation}\label{eq:pee_Introduce_KS}
\begin{cases}
\partial_t u = d_u\Delta u - \chi\nabla\cdot(u\nabla v)+f(u), \\
\partial_t v = d_v\Delta v + \alpha u - \beta v,
\end{cases}
\end{equation}
where $u(x,t)$ denotes the cell density and $v(x,t)$ denotes the concentration of the chemoattractant. The positive parameters $d_u$ and $d_v$ represent diffusion rates, $\chi$ is the chemotactic sensitivity, and $f(u):[0,\infty)\to \mathbb{R}$ is a smooth function satisfying $f(0)\geq0$, often accounting for intrinsic growth or decay.

In many biological scenarios, the timescales of cellular motion and chemical diffusion can be significantly different. Chemoattractants such as nutrients or signaling molecules often diffuse much faster than cells can migrate, reaching a quasi-steady state almost instantaneously compared to the relatively slow redistribution of the population. This scale separation motivates a reduced yet biophysically relevant formulation.

A common simplification under the assumption of fast chemical diffusion leads to the parabolic-elliptic Keller-Segel system: 
\begin{equation}\label{eq:pee_Introduce_pe}
\begin{cases}
\partial_t u = d_u\Delta u - \chi\nabla\cdot(u\nabla v)+f(u), \\
0 = d_v\Delta v + \alpha u - \beta v.
\end{cases}
\end{equation}
This version is mathematically more tractable and effectively describes long-term behavior in situations where the chemoattractant reaches equilibrium much faster than the cell population. It has been widely analyzed \cite{Biler1999global, Nagai1998global, Herrero1996, Biler1998local, Gajewski1998global} and applied to model biological phenomena such as nutrient-guided migration and stable pattern formation \cite{Keller1970initiation, Keller1971model, Horstmann}.

In many biological contexts, the relative timescales of cellular response and chemical diffusion are not fixed but may vary depending on environmental conditions or specific biological mechanisms. For instance, in certain developmental processes or under pathological conditions, the diffusion properties of signaling molecules may change, altering the dynamic coupling between cells and their chemical environment. This variability necessitates a modeling framework that can seamlessly transition between different dynamical regimes, thereby providing a more flexible tool for theoretical investigation and numerical simulation.

To bridge these diverse dynamical scenarios and incorporate variable timescale separations into a single analytical framework, one may introduce a switching parameter $\tau$ that unifies the fully parabolic and parabolic-elliptic formulations:
\begin{equation}\label{eq:pee_Introduce_tau}
\begin{cases}
\partial_t u = d_u\Delta u - \chi\nabla\cdot(u\nabla v) + f(u), \\
\tau \partial_t v = d_v\Delta v + \alpha u - \beta v,
\end{cases}
\end{equation}
where $\tau \in \{0, 1\}$. Setting $\tau=1$ recovers the fully time-dependent system \eqref{eq:pee_Introduce_KS}, while $\tau=0$ yields the parabolic-elliptic system \eqref{eq:pee_Introduce_pe}. This integrated framework facilitates the study of systems across different time-scale separations.

In many biological contexts, however, cells respond to more than one chemical signal—often integrating both attractants and repellents to navigate complex environments. This dual response mechanism is crucial for numerous physiological and ecological processes, such as immune cells coordinating movement toward inflammatory signals while avoiding inhibitory factors, or microbial populations exploring resource gradients while evading toxic regions. The ability to process competing cues enables finer spatial organization and more adaptive collective behavior than single-signal systems can capture. To model this richer, more biologically realistic behavior, the attraction-repulsion chemotaxis framework has been developed:
\begin{equation}\label{eq:pee_introduce_AR}
\begin{cases}
\partial_t u = \Delta u - \nabla \cdot (\chi u\nabla v) + \nabla \cdot (\xi u\nabla w), & \text{in } Q, \\
\tau\partial_t v = \Delta v + \alpha u - \beta v, & \text{in } Q, \\
\tau\partial_t w = \Delta w + \gamma u - \delta w, & \text{in } Q.
\end{cases}
\end{equation}
Here, $v$ represents an attractive signal (e.g., nutrients or chemoattractants) while $w$ corresponds to a repulsive one (e.g., toxins or repellents). The parameters $\chi$ and $\xi$ quantify the strengths of chemotactic attraction and repulsion, respectively, allowing the model to represent a wide spectrum of biological phenomena. Such models have been used, for instance, to describe quorum sensing in bacterial colonies, where populations modulate movement in response to self-secreted attractants and waste byproducts \cite{Painter2002volume}, as well as microglial cell behavior in Alzheimer’s disease, involving response to both amyloid-beta attractants and repulsive signals in neural tissue \cite{Luca2003chemotactic}. From a mathematical perspective, these systems exhibit rich dynamics including pattern formation, phase separation, and critical transitions. Analytical results concerning global existence, boundedness, and blow-up of solutions can be found in \cite{Lin2015large, Liu2012classical, Tao2013competing, Liu2015global, Jin2015boundedness}.

While the attraction-repulsion framework captures essential aspects of bidirectional chemical response, real biological populations are further regulated by intrinsic growth constraints and resource limitations. In vivo and in vitro experiments frequently observe that cell proliferation is self-regulated through density-dependent mechanisms, such as contact inhibition or nutrient depletion, which prevent unlimited growth and ensure population sustainability. To incorporate these critical ecological features and better align the model with experimentally observed behaviors, we extend the system by including logistic-type growth terms. This extension not only enhances the biological fidelity of the model but also introduces important mathematical structure that can suppress blow-up phenomena and promote global existence of solutions.

In this paper, we focus on an extended version of the attraction-repulsion model that includes logistic growth terms:
\begin{equation} \label{eq:pee_apply_main}
\begin{cases}
\partial_t u = \Delta u - \nabla \cdot (\chi u \nabla v) + \nabla \cdot (\xi u\nabla w) + r u - \mu u^2, & \text{in } Q, \\
\tau \partial_t v = \Delta v + \alpha u - \beta v, & \text{in } Q, \\
\tau \partial_t w = \Delta w + \gamma u - \delta w, & \text{in } Q, \\
\partial_{\nu}u = \partial_{\nu}v = \partial_{\nu}w = 0, & \text{on } \Sigma, \\
u(x,0) = f(x),\quad v(x,0) = g(x),\quad w(x,0) = h(x), & \text{in } \Omega.
\end{cases}
\end{equation}
Here, $\Omega \subset \mathbb{R}^{n}$ ($n \geq 2$) is a bounded Lipschitz domain, $Q := \Omega \times (0,\infty)$ is the space-time cylinder, and $\Sigma := \partial \Omega \times (0,\infty)$ is its lateral boundary. These notations will be used consistently throughout the paper, including in the more general formulation \eqref{eq:pee_intro} introduced below.

The logistic term $ r u - \mu u^2$ in \eqref{eq:pee_apply_main} introduces density-dependent proliferation and death, crucial for modeling population saturation and carrying capacity -- features that are fundamental to ecological modeling and essential for realistic long-term behavior. The positive parameter $r$ represents the intrinsic growth rate, while $\mu$ quantifies the strength of intra-specific competition. This extension not only enriches the biological relevance of the model but also significantly influences its analytical properties, often ensuring global existence of solutions and preventing uncontrolled growth. Well-posedness and qualitative properties of this and related systems have been studied in works such as \cite{Zhao2017parabolic}.

Let $B := \{\chi, \xi, r, \mu, \alpha, \beta, \gamma, \delta\}$ denote the set of biological parameters. We define the measurement operator:
\begin{equation}\label{eq:pee_Measurement2}
\mathcal{M}^{+}_{B}(f,g,h) = \big( (u(x,t),v(x,t),w(x,t))|_{\Sigma},\ u(\cdot,T),\ v(\cdot,T),\ w(\cdot,T) \big), \quad x \in \Omega.
\end{equation}
The associated inverse problem is then formulated as:
\begin{equation}\label{eq:pee_inverseP2}
\mathcal{M}^{+}_{B} \to B,
\end{equation}
where the superscript `$+$' emphasizes that we consider non-negative solutions biologically relevant for population densities. 
We are again concerned with the issue of unique identifiability: for two configurations $B_1$ and $B_2$, does the following hold?
\begin{equation}\label{eq:pee_UniqueIden2}
\mathcal{M}^{+}_{B_1} = \mathcal{M}^{+}_{B_2} \quad \text{if and only if} \quad B_1 = B_2.
\end{equation}
A rigorous statement of this corollary is provided in Section \ref{sec:pee_mainrs}.

\subsection{Mathematical setup}\label{sec:pee_Setup}
Building upon the biological foundation established above, we now present the precise mathematical framework investigated in this work. We consider the following coupled nonlinear parabolic-elliptic-elliptic system:
\begin{equation}\label{eq:pee_intro}
\begin{cases}
\partial_t u = \Delta u - \nabla \cdot (\chi u\nabla v) + \nabla \cdot (\xi u\nabla w) + F(x,u), & \text{in } Q, \\
0 = \Delta v + G(x,u,v), & \text{in } Q, \\
0 = \Delta w + H(x,u,w), & \text{in } Q, \\
\partial_{\nu}u = \partial_{\nu}v = \partial_{\nu}w = 0, & \text{on } \Sigma, \\
u(x,0) = f(x),\quad v(x,0) = g(x),\quad w(x,0) = h(x), & \text{in } \Omega,
\end{cases}
\end{equation}
where the domain $\Omega$, the space-time cylinder $Q$, and the lateral boundary $\Sigma$ are as defined below \eqref{eq:pee_apply_main}. The functions $F(x,m): \Omega \times \mathbb{R} \to \mathbb{R}$ and $G(x,m,z), H(x,m,z): \Omega \times \mathbb{R} \times \mathbb{R} \to \mathbb{R}$ are real-valued with respect to the variables $m$ and $z$. 

System \eqref{eq:pee_intro} represents a generalized attraction-repulsion chemotaxis model where the chemical equations are assumed to reach equilibrium rapidly relative to cellular dynamics. Biologically, $u(x,t)$ denotes the population density of cells or microorganisms, while $v(x,t)$ and $w(x,t)$ represent the concentrations of a chemoattractant and chemorepellent, respectively. The parameters $\chi > 0$ and $\xi > 0$ quantify the strengths of attraction toward $v$ and repulsion from $w$, capturing the bidirectional chemical response mechanism discussed in Subsection \ref{sec:pee_BioBack}.

More generally, the system can be formulated with variable time-scale separation through the introduction of parameters $\tau_1, \tau_2 \in \{0,1\}$:
\begin{equation}\label{eq:pee_mainuse0}
\begin{cases}
\partial_t u = \Delta u - \nabla \cdot (\chi u\nabla v) + \nabla \cdot (\xi u\nabla w) + F(x,u), & \text{in } Q, \\
\tau_1 \partial_t v = \Delta v + G(x,u,v), & \text{in } Q, \\
\tau_2 \partial_t w = \Delta w + H(x,u,w), & \text{in } Q,\\
\partial_{\nu}u = \partial_{\nu}v = \partial_{\nu}w = 0, & \text{on } \Sigma, \\
u(x,0) = f(x),\quad v(x,0) = g(x),\quad w(x,0) = h(x), & \text{in } \Omega.
\end{cases}
\end{equation}

Setting $\tau_1 = \tau_2 = 0$ yields the parabolic-elliptic-elliptic system \eqref{eq:pee_intro}, while $\tau_1 = \tau_2 = 1$ corresponds to a fully parabolic system. Our analysis encompasses both scenarios, with results applicable to this broader framework.

The central inverse problem addressed in this paper concerns the determination of unknown coefficients $\chi, \xi$ and nonlinear functions $F$, $G$, and $H$ from boundary and final-time measurements. Formally, we define the measurement operator:
\begin{equation}\label{eq:pee_Measurement1}
\mathcal{M}^{+}_{\chi,\xi,F,G,H}(f,g,h) = \big( u(x,t)|_{\Sigma}, v(x,t)|_{\Sigma}, w(x,t)|_{\Sigma}, u(\cdot,T), v(\cdot,T), w(\cdot,T) \big),
\end{equation}
where the superscript '$+$' emphasizes that we consider non-negative solutions biologically relevant for population densities. The inverse problem is then formulated as:
\begin{equation}\label{eq:pee_inverseP1}
\mathcal{M}^{+}_{\chi,\xi,F,G,H} \to \{\chi, \xi, F, G, H\}.
\end{equation}

We focus on the fundamental question of unique identifiability: for two parameter configurations $A_1 = \{\chi_1, \xi_1, F_1, G_1, H_1\}$ and $A_2 = \{\chi_2, \xi_2, F_2, G_2, H_2\}$, does

\begin{equation}\label{eq:pee_UniqueIden1}
\mathcal{M}^{+}_{A_1} = \mathcal{M}^{+}_{A_2} \quad \text{if and only if} \quad A_1 = A_2?
\end{equation}

A rigorous statement of our main identifiability result is presented in Section \ref{sec:pee_mainrs}.

For the remainder of this paper, we specialize to the system \eqref{eq:pee_mainuse0} into:
\begin{equation}\label{eq:pee_mainuse}
\begin{cases} 
\partial_t u= \Delta u- \nabla \cdot (\chi u\nabla v)+\nabla \cdot (\xi u\nabla w)+F(x,u),&\  \text{in} \    Q,\\ 
\tau \partial_t v=\Delta v+G(x,u,v), &\  \text{in} \    Q,\\
\tau \partial_t w=\Delta w+H(x,u,w), &\  \text{in} \    Q,\\
\partial_{\nu}u=\partial_{\nu}v=\partial_{\nu}w=0, &\  \text{on}\    \Sigma,\\ 
u(x,0)=f(x),\, v(x,0)=g(x) ,\, w (x,0)=h(x),  &\  \text{in} \   \Omega,\\ 
\end{cases}
\end{equation} 
where $\tau \in \{0,1\}$ unifies the parabolic-elliptic and fully parabolic cases. The nonlinear terms assume specific analytic forms:
\begin{equation}\label{eq:pee_definition_F}
   F(x,m):=rm-\mu m^2, 
\end{equation}
representing logistic growth with intrinsic rate $r > 0$ and carrying capacity parameter $\mu > 0$. 
The functions $G(x,m,z), H(x,m,z)$ are analytic with respect to $m$ and $z,$ and are of the forms below:
\begin{equation}\label{eq:pee_definition_GH}
 G(x,m,z):=\sum\limits_{p,q=0,p+q>0}\alpha^{pq}(x)m^{p}z^q, \quad H(x,m,z):=\sum\limits_{r,s=0,r+s>0}\beta^{rs}(x)m^{r}z^s.
\end{equation}
This formulation maintains biological relevance while providing sufficient mathematical structure for rigorous analysis of the inverse problem.

We emphasize that the unknown parameters consist of both constants and spatially dependent functions. Specifically, the chemotactic sensitivities $\chi$ and $\xi$, the logistic growth coefficients $r$ and $\mu$, and the first-order decay rates $\alpha^{01}$ and $\beta^{01}$ are positive constants. In contrast, the production rates $\alpha^{10}(x)$ and $\beta^{10}(x)$, as well as the higher-order interaction coefficients $\alpha^{pq}(x)$ and $\beta^{rs}(x)$ for $p+q \ge 2$ and $r+s \ge 2$, are functions of the spatial variable $x$. The precise structure and admissibility conditions for the coefficients $\alpha^{pq},\beta^{rs}$ are specified in Definitions \ref{defi:pee_admissibleG} and \ref{defi:pee_admissibleH}.

Furthermore, we clarify that the measurement operator $\mathcal{M}^{+}_{\chi,\xi,F,G,H}$ corresponds to a full set of boundary and final-time observations generated by varying the initial data $(f,g,h)$ over an appropriate space of admissible non-negative functions. This framework, often referred to as a Dirichlet-to-Neumann map, provides infinitely many measurements, which are necessary for the high-order variation method. The ability to vary the initial data is essential to our approach, as it allows us to isolate individual coefficients by selectively activating or suppressing specific terms in the linearized systems.

We emphasize that our results do not address the single-measurement case, where only one fixed initial datum $(f_0,g_0,h_0)$ is prescribed and the corresponding boundary and final-time data are observed. The single-measurement inverse problem is considerably more challenging and requires fundamentally different techniques. In particular, the high-order variation method is no longer applicable, as it relies on the availability of infinitely many measurements corresponding to a family of initial data. To treat the single-measurement case for coupled nonlinear parabolic-elliptic-elliptic systems, one would typically need to employ Carleman estimates to derive stability and uniqueness results from a single observation. We refer to the foundational works of M. Klibanov \cite{Klibanov1992IPCarleman,BukhgeimKlibanov1981Carleman,KlibanovTimonov2004CarlemanEstBook,Klibanov2013JIIPCarlemanReview}, O. Imanuvilov, M. Yamamoto \cite{Yamamoto2009IPCarlemanParabolic, IY1998} and their collaborators for the development of Carleman estimate-based methodologies for inverse problems with single measurements. These techniques have been successfully applied to various coefficient inverse problems for parabolic and hyperbolic equations; however, extending them to the mixed parabolic-elliptic-elliptic chemotaxis system considered here remains a substantial open problem due to the nonlinear coupling between equations of different types and the inherent analytical challenges of mixed-type systems. The interplay between time-dependent parabolic dynamics and instantaneous elliptic equilibrium constraints introduces fundamental difficulties that are absent in purely parabolic or purely elliptic settings, making the single-measurement problem for such mixed-type systems particularly challenging.

The primary novelty of our work lies in establishing unique identifiability results for this class of coupled nonlinear chemotaxis systems, addressing a significant gap in the inverse problems literature for biologically relevant PDE models with mixed parabolic-elliptic dynamics.

\subsection{Statement of main results} \label{sec:pee_mainrs}
This section is devoted to establishing the mathematical framework and presenting the main results. We begin by introducing several admissible classes for the nonlinearities $F$, $G$, and $H$ that appear in systems \eqref{eq:pee_mainuse}. These classes ensure the necessary analytic and structural conditions required for our subsequent analysis. The definitions generalize previous frameworks by allowing expansion around arbitrary non-negative constant solutions and incorporating spatial dependence through multiplicative separability conditions.

Suppose $(u_0,v_0,w_0)$ is a known non-negative constant solution of \eqref{eq:pee_mainuse} and $F,G,H$ are analytic. Then, we can introduce the following admissible classes.

\begin{defi}\label{defi:pee_admissibleF}
Let $(u_0,v_0,w_0)$ be a known non-negative constant solution of \eqref{eq:pee_mainuse}. We say that $U(x,z): \mathbb{R}^n \times \mathbb{C} \to \mathbb{C}$ is admissible, denoted by $U \in \mathcal{A},$ if:

(a) The map $z \mapsto U(\cdot,  z) $ is holomorphic with value in $C^{2+\alpha}_0(\bar{\Omega}),$

(b) $U(x,u_0)=0$ for all $x \in \Omega$.

It is clear that if $U$ satisfies these two conditions, it can be expanded into a power series
\[U(x,z)=\sum\limits^{\infty}_{m=1}U_{m}(x)\frac{z^m }{m!},\]
where $U_{m}(x)=\frac{\partial^m}{\partial z^m}U(x,u_0)\in C^{2+\alpha}_0(\Omega)$.
\end{defi}

This admissibility condition is imposed by analytically extending $U$ to a holomorphic function $\tilde{U}$ in the complex variable $z$ and then restricting back to the real line. Since $U$ is real-valued for real arguments, we assume the coefficients $U_m(x)$ are real-valued.

We next define a specific spatial structure required for the subsequent analysis.

\begin{defi}
    We say that a function $A(x)$ is of \emph{multiplicative separable form} (with respect to the $n$-th spatial variable), if \[A(x)=A_1(x_1,\dots,x_{n-1})A_2(x_n),\]
for $x=(x_1,\dots,x_n)$ and \[\int_{\{x_n:(x_1,\dots,x_n)\in\Omega\}} A_2(x_n)dx_n\neq 0.\]
\end{defi}

\begin{defi}\label{defi:pee_admissibleG}
Let $(u_0,v_0,w_0)$ be a known non-negative constant solution of \eqref{eq:pee_mainuse}.  We say that $V(x,p,q): \mathbb{R}^n \times \mathbb{C}  \times \mathbb{C} \to \mathbb{C}$ is admissible, denoted by $V \in \mathcal{B},$ if: 

(a) The map $(p,q) \mapsto V(\cdot, p,q) $ is holomorphic with value in $C^{2+\alpha}_0(\bar{\Omega}),$

(b) $V(x,u_0,v_0)=0$ for all $x \in \Omega,$ 

(c) The first-order Taylor coefficient $V^{(0,1)}(x,\cdot,v_0)$ is constant,

(d) The first-order Taylor coefficient $V^{(1,0)}(x,u_0,\cdot)$ is independent of one variable,

(e) The higher-order Taylor coefficient $V^{(k)}(x,\cdot,\cdot) (k\geq 2)$ is of a multiplicative separable form for $x\in \mathbb{R}^{n}$.

It is clear that if $V$ satisfies these four conditions, it can be expanded into a power series
\[V(x,p,q)=\sum\limits^{\infty}_{m\geq1,\, n\geq 0}V_{mn}(x)\frac{p^m q^n}{(m+n)!},\]
where $V_{mn}(x)=\frac{\partial^m}{\partial p^m}\frac{\partial^n}{\partial q^n}V(x,u_0,v_0) \in C^{2+\alpha}_0(\Omega)$.
\end{defi}

Similarly, we can give admissible class for $H(x,u,w)$ as:
\begin{defi}\label{defi:pee_admissibleH}
Let $(u_0,v_0,w_0)$ be a known non-negative constant solution of \eqref{eq:pee_mainuse}.  We say that $W(x,p,q): \mathbb{R}^n \times \mathbb{C}  \times \mathbb{C} \to \mathbb{C}$ is admissible, denoted by $W \in \mathcal{C},$ if: 

(a) The map $(p,q) \mapsto W(\cdot, p,q) $ is holomorphic with value in $C^{2+\alpha}_0(\bar{\Omega}),$

(b) $W(x,u_0,w_0)=0$ for all $x \in \Omega,$

(c) The first-order Taylor coefficient $W^{(0,1)}(x,\cdot,w_0)$ is constant,

(d) The first-order Taylor coefficient $W^{(1,0)}(x,u_0,\cdot)$ is independent of one variable,

(e) The higher-order Taylor coefficient $W^{(k)}(x,\cdot,\cdot) (k\geq 2)$ is of a multiplicative separable form for $x\in \mathbb{R}^{n}$.

It is clear that if $W$ satisfies these four conditions, it can be expanded into a power series
\[W(x,p,q)=\sum\limits^{\infty}_{m\geq1,\, n\geq 0}W_{mn}(x)\frac{p^m q^n}{(m+n)!},\]
where $W_{mn}(x)=\frac{\partial^m}{\partial p^m}\frac{\partial^n}{\partial q^n}W(x,u_0,w_0) \in C^{2+\alpha}_0(\Omega)$.
\end{defi}

It can be easily seen that for $F\in \mathcal{A},G\in \mathcal{B}$ and $H\in \mathcal{C},$ $F,G$ and $H$ are of the forms \eqref{eq:pee_definition_F} and \eqref{eq:pee_definition_GH} respectively.

These definitions generalize the admissibility conditions used in previous biological inverse problems \cite{LLL2024inverse, LL2024determining}. Unlike \cite{LL2024determining}, which expands around $(0,0)$, our framework allows for expansion around any constant solution $(u_0, v_0, w_0)$. Furthermore, compared to \cite{LLL2024inverse} which requires constant Taylor coefficients, we allow the coefficients $V_{mn}(x)$ and $W_{mn}(x)$ to be functions of $x$, subject to the separability condition, thereby broadening the applicability of our results.

Under these admissibility conditions, we can establish the well-posedness of \eqref{eq:pee_mainuse}. We consider a more general system:
\begin{equation}\label{eq:pee_wellposed_use}
\begin{cases} 
\partial_t u= \Delta u- \nabla \cdot (\chi u\nabla v)+\nabla \cdot (\xi u\nabla w)+h(u),&\  \text{in} \    Q,\\ 
\tau \partial_t v=\Delta v+b_1(x,u,v), &\  \text{in} \    Q,\\
\tau \partial_t w=\Delta w+b_2(x,u,w), &\  \text{in} \    Q,\\
\partial_{\nu}u=\partial_{\nu}v=\partial_{\nu}w=0, &\  \text{on}\    \Sigma,\\ 
u(x,0)=f(x),\, v(x,0)=g(x) ,\, w (x,0)=h(x),  &\  \text{in} \   \Omega.\\ 
\end{cases}
\end{equation} 

When $\tau = 1$, the system \eqref{eq:pee_wellposed_use} is fully parabolic. Choosing $b_1(x,u,v) = u - v$ and $b_2(x,u,w) = u - w$, it is known from \cite{jiao2024global} that if $h(u)$ satisfies
\[
h \in C^1([0,\infty)), \quad h(0) \geq 0, \quad h(s) \leq r - \mu s^{\gamma} \text{ for } s \geq 0, \quad \mu, \gamma > 0, \quad r \geq 0,
\]
then \eqref{eq:pee_wellposed_use} possesses a global classical solution $(u,v,w)$ under certain conditions (see Theorem 1.1 in \cite{jiao2024global}). 

When $\tau = 0$, the system \eqref{eq:pee_wellposed_use} becomes parabolic-elliptic-elliptic. With $b_1(x,u,v) = u - v$, $b_2(x,u,w) = u - w$, and $h(u) = 0$, local and global well-posedness results can be established based on Theorems 1.1 and 1.2 in \cite{shi2015well}. For the logistic case $h(u) = r u - \mu u^2$ and with $f \in C^{2+\alpha}(\Omega)$, a unique local solution exists. Modifying $b_1$ and $b_2$ to $b_1(x,u,v) = \alpha u - \beta v$ and $b_2(x,u,w) = \gamma u - \delta w$, Theorem 1.1 in \cite{Zhao2017parabolic} guarantees a unique, uniformly bounded global classical solution $(u,v,w)$ to \eqref{eq:pee_wellposed_use}.

These results provide the necessary well-posedness foundation for the applied model \eqref{eq:pee_apply_main}.
We now present the main results of this paper concerning the inverse problems. Our goal is to uniquely determine the unknown biological parameters—including the chemotactic sensitivities $\chi$ and $\xi$, the logistic growth coefficients $r$ and $\mu$, and the coefficients $\alpha^{pq}$, $\beta^{pq}$ (for $p,q \geq 0$, $p+q>0$) governing the chemical kinetics—from the measurement operator $\mathcal{M}^{+}_{A}$.

\begin{thm} \label{thm:pee_mainthm}
Suppose that the system \eqref{eq:pee_mainuse} has a solution $(u,v,w)\in C^{1+\frac{\alpha}{2},2+\alpha}_0(Q)\times C^{1+\frac{\alpha}{2},2+\alpha}_0(Q)\times C^{1+\frac{\alpha}{2},2+\alpha}_0(Q)$, and coefficients $\chi$ and $\xi$ in \eqref{eq:pee_mainuse}, as well as all orders of $u$, $v$ and $w$, defined under the notion of high-order variation, are independent of one spatial variable.  Assume $F \in \mathcal{A},$ $G \in \mathcal{B}$ and $H\in\mathcal{C}$ for $j=1,2. $ Let $\mathcal{M}^{+}_{A}$ be the associated measurement map as defined in \eqref{eq:pee_Measurement1}. For any $(f,g,h)\in C^{2+\alpha}(\Omega)\times  C^{2+\alpha}(\Omega)\times C^{2+\alpha}(\Omega)$, one has  \[\mathcal{M}^{+}_{A_1}(f,g,h)=\mathcal{M}^{+}_{A_2}(f,g,h),\]
then it holds that 
\[ A_1=A_2  \ \text{in}\  \, \Omega \times \mathbb{R}.\]
\end{thm}

Based on the well-posedness results discussed above, we derive the following corollary for the specific biological system \eqref{eq:pee_apply_main}.
\begin{cor}\label{cor:pee_appcor}
 Let $\mathcal{M}^{+}_{B_j}$ $(B_j=\{\chi_j,\xi_j,r_j,\mu_j,\alpha_j,\beta_j,\gamma_j,\delta_j\},j=1,2)$ be the measurement map associated to the following system:
\begin{equation} \label{eq:pee_apply_prop}
\begin{cases} 
\partial_t u_j= \Delta u_j- \nabla \cdot (\chi_j u_j \nabla v_j)+\nabla \cdot (\xi_j u_j\nabla w_j)+r_ju-\mu_j u^2,&\  \text{in} \    Q,\\ 
\tau \partial_t v_j=\Delta v_j+\alpha_j u_j-\beta_j v_j, &\  \text{in} \    Q,\\
\tau \partial_t w_j=\Delta w_j+\gamma_j u_j-\delta_j w_j, &\  \text{in} \    Q,\\
\partial_{\nu}u_j=\partial_{\nu}v_j=\partial_{\nu}w_j=0, &\  \text{on}\    \Sigma,\\ 
u_j(x,0)=f(x),\, v_j(x,0)=g(x) ,\, w_j(x,0)=h(x),  &\  \text{in} \   \Omega.
\end{cases} 
\end{equation} 
For a given set of parameters $\beta_j,\delta_j,\chi_j$, and $\xi_j$, suppose there exists a solution $(u_j,v_j,w_j)$ to \eqref{eq:pee_apply_prop} that also satisfies the assumptions of Theorem \ref{thm:pee_mainthm}. If for any $f,g,h \in C^{2+\alpha}(\Omega),$ one has  \[\mathcal{M}^{+}_{B_1}(f,g,h)=\mathcal{M}^{+}_{B_2}(f,g,h),\]
then it holds that 
\[ B_1=B_2  \ \text{in}\  \, \Omega \times \mathbb{R}.\]
\end{cor}
 
\subsection{Technical developments and discussion}\label{sec:pee_Techn}
Our work presents a significant breakthrough by establishing a novel theoretical framework for solving inverse problems in complex biological systems modeled by coupled nonlinear partial differential equations of mixed parabolic-elliptic type.

A primary novelty of our work lies in the first systematic investigation of unique identifiability for a coupled nonlinear system of mixed parabolic-elliptic-elliptic type using only boundary measurements. Previous studies have predominantly focused on systems that are either purely parabolic \cite{choulli2018,isakov1993} or purely elliptic \cite{lassas2021inverse,lassas2020partial}, leaving a significant gap for hybrid models that arise in complex biological applications.

This mixed-type structure introduces extraordinary theoretical challenges that render conventional techniques inadequate. Unlike purely parabolic or elliptic systems, the coupled model requires fundamentally new approaches to handle the interplay between time-dependent dynamics and instantaneous equilibrium constraints. The inherent analytical complexities are substantial: the nonlinear coupling between equations of different types prevents direct application of standard methods such as energy estimates or perturbation techniques.

A major technical difficulty arises from the fact that while the functions $v(x,t)$ and $w(x,t)$ satisfy elliptic equations, they maintain crucial time dependence through their coupling with the parabolic equation for $u(x,t)$. This temporal dependence directly contradicts the construction of traditional Complex Geometric Optics (CGO) solutions, which typically rely on purely spatial harmonic functions.

To overcome these challenges, we develop a novel methodology that significantly extends existing techniques. Inspired by \cite{liu2015determining}, we introduce a framework capable of handling equations with multiplicative separable spatial structure, particularly coefficients of the form $a(x) = a_1(x_1, \cdots, x_{n-1}) a_2(x_n)$ as formalized in Lemma \ref{lem:pee_MultiSepe}. A key innovation is the combined use of the fundamental theorem of calculus and the inverse Fourier transform, which enables effective decoupling of temporal and spatial dependencies. This approach provides a unified analytical framework for simultaneous recovery of all unknown model parameters through strategic variation of initial input data, representing a substantial advancement beyond existing methods for inverse problems.

Building upon our previous work on inverse problems in biological systems, this study shifts focus from macroscopic population dynamics to the microscale processes of chemotaxis—where individual cells navigate chemical gradients. While our prior investigations \cite{DL2024inverse,LL2024determining,LLL2024inverse,LLL2025inverse,li2025turing} addressed system-level behaviors using purely parabolic models: \cite{LLL2024inverse} introduced a high-order variation method to ensure physiologically meaningful solutions and recover source coefficients under admissible classes, and \cite{LLL2025inverse} focused on multi-population aggregation systems with recovery of diffusion rates, advection coefficients, and interaction kernels, the current work enters a less charted territory in inverse problems: cell-level mechanism identification. Rather than examining collective outcomes, we target the fundamental signaling and response mechanisms that govern individual cell behavior, an area that remains largely open in the field of parameter identification.

In \cite{li2024simultaneous}, we began exploring environmental interaction in a 3D chemotaxis-fluid model, though within a purely parabolic framework and with limited parameter recovery. Here, we advance into biologically more realistic settings by incorporating mixed parabolic-elliptic dynamics, which better represent the rapid diffusion of chemical signals relative to cell movement.

This extension is not only mathematically novel but also physiologically critical: it allows us to model how cells process simultaneous attractive and repulsive cues in realistic microenvironments. As in all our studies, we rigorously preserve solution non-negativity to ensure that results remain consistent with biological constraints. By bridging microscopic cell behavior and population-level outcomes, this work offers a more comprehensive mathematical framework for understanding gradient-guided migration in development, immunity, and disease.

The remainder of this paper is organized as follows. In Section \ref{sec:pee_pfmain}, we present the detailed proof of the main uniqueness theorem (Theorem \ref{thm:pee_mainthm}). Section \ref{sec:pee_application} is devoted to the proof of the corollary (Corollary \ref{cor:pee_appcor}) for the specific biological application model. Additional technical lemmas and supporting results are included in the subsequent sections.

{\centering \section{PROOF OF THE MAIN THEOREM}   \label{sec:pee_pfmain}}
This section is devoted to the proof of the main theorem. We begin by establishing two key auxiliary lemmas that play a fundamental role in our analysis.

\subsection{Auxiliary lemmas} \label{sec:pee_lemmas}
 The first lemma provides a spectral representation of solutions to a linear parabolic system, which will be used to construct specific input data for our inverse problem.

\begin{lem}\label{lem:pee_solForm}
Consider the system
\begin{equation}\label{eq:pee_lemma}
\begin{cases} 
\partial_t u(x,t)-q\Delta u(x,t)+k u(x,t)= 0, &\  \text{in} \   Q,\\ 
\partial_{\nu} u(x,t)=0, &\  \text{on} \   \Sigma,
\end{cases} 
\end{equation}
where $q$ and $k$ are constants. There exists a sequence of solutions $u(x,t)$ to \eqref{eq:pee_lemma} such that 
\begin{enumerate}
    \item $u(x,t)=e^{\theta t} l(x;\theta)$ for some $\theta \in \mathbb{R}^n$ and $l(x;\theta) \in C^{2}(\Omega).$ Notably, $l(x;\theta)$ is not necessarily 0, and $\frac{\theta}{q}$ is its corresponding eigenvalue;
    \item There does not exist an open subset $U$ of $\Omega$ such that $\nabla l(x;\theta)=0$ in $U$.
\end{enumerate}
\end{lem}

\textit{Proof.}  See \cite{LLL2024inverse}.

The second lemma establishes a uniqueness result for functions with multiplicative separable structure, which is essential for recovering spatially dependent coefficients.

\begin{lem}\label{lem:pee_MultiSepe}
    Let $f\in C^{2+\alpha}(\Omega)$ be a function that is independent of one variable, say $x_n$; that is, $f(x)=f(x_1,\cdots,x_{n-1})\in \mathbb{R}^{n}$. Take
\begin{equation*}
    \xi = (0,\cdots, 0,\xi_{n})+\mathrm{i}(\xi_1,\cdots,\xi_{n-1},0)=:{\xi^{\prime}}^{\perp}+\mathrm{i}(\xi^{\prime},0),
\end{equation*}
with $\vert \xi_{n}\vert=\vert \xi^{\prime}\vert$, and let $\Phi(x)=e^{\xi\cdot x}$. Consider two functions $f(x)=\alpha(x_1,\cdots,x_{n-1})\beta(x_n)$ and $\tilde{f}(x)=\tilde{\alpha}(x_1,\cdots,x_{n-1})\tilde{\beta}(x_n)$, where
    \[\int_{\{x_n:(x_1,\dots,x_n)\in\Omega\}}\beta(x_n)dx_n=\int_{\{x_n:(x_1,\dots,x_n)\in\Omega\}}\tilde{\beta}(x_n)dx_n\neq 0.\]
  If both $f(x)$ and $\tilde{f}(x)$ satisfy $\int_{\Omega}f(x)\Phi(x)dx=0$ and $\int_{\Omega}\tilde{f}(x)\Phi(x)dx=0$, then $f(x)=\tilde{f}(x)$.
\end{lem}
\begin{proof}
The proof of the first part can be found in \cite{liu2015determining}. Here, we focus on the multiplicative separable case. Take
\begin{equation*}
    \xi = (0,\cdots, 0,\xi_{n})+\mathrm{i}(\xi_1,\cdots,\xi_{n-1},0)=:{\xi^{\prime}}^{\perp}+\mathrm{i}(\xi^{\prime},0),
\end{equation*}
with $\vert \xi_{n}\vert=\vert \xi^{\prime}\vert$, and for the harmonic function $\Phi$, we take $\Phi(x)=e^{\xi\cdot x}$.  Assume that there exist $f$ and $\tilde{f}$ such that $\int_{\Omega}f(x)\Phi(x)dx=\int_{\Omega}\tilde{f}(x)\Phi(x)dx$, which implies
\begin{equation}\notag
    \int_{\Omega}\left(\alpha(x^{\prime})\beta(x_n)-\tilde{\alpha}(x^{\prime})\tilde{\beta}(x_n))\right)e^{\xi \cdot x}dx=0,
\end{equation}
where $x^{\prime}=(x_1,\cdots,x_{n-1})$. 
Separating variables yields
\begin{align}\notag
& \int_{\{x^{\prime}:(x^{\prime},x_n)\in\Omega\}} \alpha(x^{\prime})e^{\mathrm{i}\xi^{\prime}\cdot x^{\prime}}dx^{\prime}\int_{\int_{\{x_n:(x^{\prime},x_n)\in\Omega\}}}\beta(x_n)e^{\xi_{n} x_n}dx_n=\\
& \qquad\qquad\qquad\qquad\int_{\{x^{\prime}:(x^{\prime},x_n)\in\Omega\}} \tilde{\alpha}(x^{\prime})e^{\mathrm{i}\xi^{\prime}\cdot x^{\prime}}dx^{\prime}\int_{\int_{\{x_n:(x^{\prime},x_n)\in\Omega\}}}\tilde{\beta}(x_n)e^{\xi_{n} x_n}dx_n, 
\end{align}
where the integrals in $\alpha(x^{\prime})$ and $\tilde{\alpha}(x^{\prime})$ are clearly their Fourier transforms. Thus, we can rewrite this as
\begin{equation}\notag
    \hat{\alpha}(\xi^{\prime})\int_{\{x_n:(x^{\prime},x_n)\in\Omega\}}\beta(x_n)e^{\xi_{n}x_n}dx_n=\hat{\tilde{\alpha}}(\xi^{\prime})\int_{\{x_n:(x^{\prime},x_n)\in\Omega\}}\tilde{\beta}(x_n)e^{\xi_{n}x_n}dx_n. 
\end{equation}

Denote $\text{dim}(\Omega)=R.$ For $\vert \xi^{\prime}\vert<\delta<\frac{1}{2R}$ sufficiently small, we conduct a series expansion for $e^{\xi_{n} x_n}$ and have
\begin{equation}\label{eq:pee_expansionG}
    \hat{\alpha}(\xi^{\prime})\sum\limits^{\infty}_{j=0}\Gamma_{\beta,j}\vert \xi^{\prime}\vert^j=\hat{\tilde{\alpha}}(\xi^{\prime})\sum\limits^{\infty}_{j=0}\Gamma_{\tilde{\beta},j}\vert \xi^{\prime}\vert^j,
\end{equation}
where \[\Gamma_{\beta,j}=\frac{1}{j!}\int_{\{x_n:(x^{\prime},x_n)\in\Omega\}}\beta(x_n)x^j_ndx_n, \Gamma_{\tilde{\beta},j}=\frac{1}{j!}\int_{\{x_n:(x^{\prime},x_n)\in\Omega\}}\tilde{\beta}(x_n)x^j_ndx_n,\quad j=0,1,\dots.\]
This expansion can be written as
\begin{equation}\label{eq:pee_expansionG2}
     \hat{\alpha}(\xi^{\prime}) \Gamma_{\beta,0}+\hat{\alpha}(\xi^{\prime})\sum\limits^{\infty}_{j=1}\Gamma_{\beta,j}\vert \xi^{\prime}\vert^j=\hat{\tilde{\alpha}}(\xi^{\prime})\Gamma_{\tilde{\beta},0}+\hat{\tilde{\alpha}}(\xi^{\prime})\sum\limits^{\infty}_{j=1}\Gamma_{\tilde{\beta},j}\vert \xi^{\prime}\vert^j.
\end{equation}
Comparing coefficients, we find $\hat{\alpha}(\xi^{\prime})\Gamma_{\beta,0}=\hat{\tilde{\alpha}}(\xi^{\prime})\Gamma_{\tilde{\beta},0}$. By the assumption of the lemma, $\Gamma_{\beta,0}=\Gamma_{\tilde{\beta},0}\neq 0$, so $\hat{\alpha}(\xi^{\prime})=\hat{\tilde{\alpha}}(\xi^{\prime})$. Since this holds for all $\xi^{\prime}$, we conclude $\alpha=\tilde{\alpha}$ by the inverse Fourier transform. 

Next, comparing the coefficients of higher order terms of $\vert \xi^{\prime}\vert$, since we have recovered $\alpha$, we obtain  $\Gamma_{\beta,j}=\Gamma_{\tilde{\beta},j}$ for each $j\in\mathbb{N}$. This implies  that the moment functions associated to $\beta$ and $\tilde{\beta}$ are identical, thus implying that $\beta(x_n)=\tilde{\beta}(x_n)$ almost everywhere.

Hence, the lemma is proved. 
\end{proof}

\subsection{Recovery of the first-order coefficients} \label{sec:pee_mainR1}
We now proceed to the proof of our main theorem. In this proof, we focus solely on the case $\tau=0$ because inverse problems for fully parabolic systems in biological models have already been studied in works such as \cite{LL2024determining,LLL2024inverse,LLL2025inverse}. The proof for the case $\tau = 1$ is briefly demonstrated in the proof of the corollary.

For $j=1,2$, consider the system:
\begin{equation} \label{eq:pee_mainpfuse}
\begin{cases} 
\partial_t u_{j}= \Delta u_j- \nabla \cdot (\chi_j u_j \nabla v_j)+\nabla \cdot (\xi_j u_j\nabla w_j)+r_ju_j-\mu_j u^2_j,&\  \text{in} \    Q,\\ 
0=\Delta v_j+\sum\limits_{p,q=0,p+q>0}\alpha^{pq}_{j}u^{p}_{j}v^q_j, &\  \text{in} \    Q,\\
0=\Delta w_j+\sum\limits_{r,s=0,r+s>0}\beta^{rs}_{j}u^{r}_{j}w^s_j, &\  \text{in} \    Q,\\
\partial_{\nu}u_{j}=\partial_{\nu}v_{j}=\partial_{\nu}w_{j}=0, &\  \text{on}\    \Sigma,\\ 
u_j (x,0)=f(x),\, v_j (x,0)=g(x) ,\, w_j (x,0)=h(x),  &\  \text{in} \   \Omega.\\ 
\end{cases} 
\end{equation}

We begin by constructing high-order variation forms for the solutions $u_j(x,t), v_j(x,t)$ and $w_j(x,t)$, which yield a sequence of linearized systems derived from the asymptotic expansion of the original nonlinear system \eqref{eq:pee_mainpfuse} around a known constant equilibrium solution $(u_0,v_0,w_0)$. Specifically, for a sufficiently small positive constant $\varepsilon$, we expand the initial data functions as follows:
\[f(x;\varepsilon)=u_{0}+\varepsilon f_{1}(x)+\frac{1}{2}\varepsilon^2 f_{2}(x)+ \tilde{f}(x;\varepsilon),\]
\[g(x;\varepsilon)=v_{0}+\varepsilon g_{1}(x)+\frac{1}{2}\varepsilon^2 g_{2}(x)+ \tilde{g}(x;\varepsilon),\]
and
\[h(x;\varepsilon)=w_{0}+\varepsilon h_{1}(x)+\frac{1}{2}\varepsilon^2 h_{2}(x)+ \tilde{h}(x;\varepsilon),\]
where $0\leq f_{1}, f_{2},  g_1,g_2,h_1,h_2\in [C^{2+\alpha}(\Omega)]^N$, and $\tilde{f}(x;\epsilon),\tilde{g}(x;\epsilon),\tilde{h}(x;\epsilon)$ satisfy
\[\frac{1}{|\varepsilon|^3}\norm{\tilde{f}(x;\epsilon)}_{[C^{2+\alpha}(\Omega)]^N}=\frac{1}{|\varepsilon|^3}\norm{f(x;\varepsilon)-u_{0}-\varepsilon f_{1}(x)-\frac{1}{2}\varepsilon^2 f_{2}(x)}_{[C^{2+\alpha}(\Omega)]^N}\to0,
\]
\[\frac{1}{|\varepsilon|^3}\norm{\tilde{g}(x;\epsilon)}_{[C^{2+\alpha}(\Omega)]^N}=\frac{1}{|\varepsilon|^3}\norm{g(x;\varepsilon)-v_{0}-\varepsilon g_{1}(x)-\frac{1}{2}\varepsilon^2 g_{2}(x)}_{[C^{2+\alpha}(\Omega)]^N}\to0,
\] 
\[\frac{1}{|\varepsilon|^3}\norm{\tilde{h}(x;\epsilon)}_{[C^{2+\alpha}(\Omega)]^N}=\frac{1}{|\varepsilon|^3}\norm{h(x;\varepsilon)-w_{0}-\varepsilon h_{1}(x)-\frac{1}{2}\varepsilon^2 h_{2}(x)}_{[C^{2+\alpha}(\Omega)]^N}\to0,
\]and both the convergences are uniformly in $\varepsilon$, where
$\varepsilon\in\mathbb{R}_+$ and  
$|\varepsilon|$ small enough. The choice of the initial functions to be positive can easily be ensured as we choose our boundary measurements.

By well-posedness, we know that there exists a unique solution $(u_j(x;\varepsilon),v_j(x;\varepsilon),w_j(x;\varepsilon))$ of \eqref{eq:pee_mainuse} and $(u_j(x;\varepsilon),v_j(x;\varepsilon),w_j(x;\varepsilon))=(0,0,0)$ is the solution of \eqref{eq:pee_mainuse} when $\varepsilon=0$.
Now we define the first-order variation form for this system. 

Let $S$ be the solution operator of \eqref{eq:pee_mainuse}. Then there exists a bounded linear operator $L$ from $\mathcal{H}:=[B_{\delta}( C^{2+\alpha}(\partial\Omega))]^2$ to $[C^{2+\alpha}(\Omega)]^2$ such that
\begin{equation}
	\lim\limits_{\|(f,g,h)\|_{\mathcal{H}}\to0}\frac{\|S(f,g,h)-S(u_0,v_0,w_0)- L(f,g,h)\|_{[C^{2+\alpha}(E)]^2}}{\|(f,g,h)\|_{\mathcal{H}}}=0,
\end{equation} 
where $\|(f,g,h)\|_{\mathcal{H}}:=\|f\|_{ C^{2+\alpha}(\partial\Omega)      }+\|g\|_{C^{2+\alpha}(\partial\Omega)}+\|h\|_{C^{2+\alpha}(\partial\Omega)}$.

Now we consider $\varepsilon=0$.
Then it is easy to check that $L(f,g,h)|_{\varepsilon=0}$ is the solution map of the following system:
\begin{equation}\label{eq:pee_pf1st}
    \begin{cases} 
\partial_t u^{(I)}_{j}(x,t)=\Delta u^{(I)}_j(x,t)+r_j u^{(I)}_j(x,t), &\  \text{in} \    Q,\\ 
0=\Delta v^{(I)}_j (x,t)+\alpha^{10}_j(x)u^{(I)}_j(x,t)+\alpha^{01}_jv^{(I)}_j(x,t),\, &\  \text{in} \    Q,\\ 
0=\Delta w^{(I)}_j (x,t)+\beta^{10}_j(x)u^{(I)}_j(x,t)+\beta^{01}_jw^{(I)}_j(x,t), &\  \text{in} \    Q,\\ 
\partial_{\nu}u^{(I)}_{j}(x,t)=\partial_{\nu}v^{(I)}_{j}(x,t)=\partial_{\nu}w^{(I)}_{j}(x,t)=0 ,&\  \text{on}\    \Sigma,\\ 
u^{(I)}_j (x,0)=f_1(x),\, v^{(I)}_j(x,0)=g_1(x),\, w^{(I)}_j(x,0)=h_1(x), &\  \text{in} \   \Omega.\\ 
\end{cases} 
\end{equation}
This system is called the first-order linearization system. The positivity of $(u^{(I)}_j,v^{(I)}_j,w^{(I)}_j)$ are ensured by the non-negative of $(f_1,g_1,h_1)$. In the following, we define \[ (u^{(I)}_j, v^{(I)}_j, w^{(I)}_j):=L(f,g,h)\Big|_{\varepsilon=0}. \]
For notational convenience, we write
\[u^{(I)}_j=\partial_{\varepsilon}u_j(x;\varepsilon)|_{\varepsilon=0}, \, v^{(I)}_j=\partial_{\varepsilon}v_j(x;\varepsilon)|_{\varepsilon=0},\,\text{and}\, w^{(I)}_j=\partial_{\varepsilon}w_j(x;\varepsilon)|_{\varepsilon=0}.\]
In our subsequent discussion, we use these notations to simplify the presentation, and their intended meaning will be clear within the given context.

\textbf{Recovery of rate of proliferation $r$.} First, we show $r_1=r_2$. Let $\bar{u}^{(I)}(x,t):=u^{(I)}_1 (x,t)-u^{(I)}_2(x,t)$. From \eqref{eq:pee_pf1st} and the condition $\mathcal{M}^{+}_{A_1} = \mathcal{M}^{+}_{A_2}$, we derive the following system for $\bar{u}^{(I)}$:
\begin{equation}\label{eq:pee_bar_u1}
    \begin{cases} 
\partial_t \bar{u}^{(I)}(x,t)-\Delta \bar{u}^{(I)}(x,t) = r_1\bar{u}^{(I)}(x,t)+(r_1-r_2)u^{(I)}_2(x,t), &\  \text{in} \   Q,\\ 
\partial_{\nu}  \bar{u}^{(I)}(x,t)= \bar{u}^{(I)}(x,t)=0, &\  \text{on} \   \Sigma,\\ 
 \bar{u}^{(I)}(x,0)=\bar{u}^{(I)}(x,T)=0,& \  \text{in}\    \Omega. 
\end{cases} 
\end{equation}

Let $\omega$ be a solution of the system 
\begin{equation}\label{eq:pee_omega1}
-\partial_{t }\omega -\Delta \omega -r_1 \omega=0 \  \text{in} \  Q,
\end{equation}
where $r_1$ is an unknown constant. A CGO solution for $\omega$ to \eqref{eq:pee_omega1} is given by:
\begin{equation}\label{eq:pee_CGO_w}
\omega=e^{(|\xi|^2-r_1)t-\mathrm{i}\zeta \cdot x}, \text{ with } \mathrm{i}=\sqrt{-1} \text{ for } \zeta \in \mathbb{R}^n.
\end{equation}

Note that $u^{(I)}_2(x,t)$ in \eqref{eq:pee_pf1st} satisfies the form of \eqref{eq:pee_lemma} with $q=1$, $k=-r_2$, and $r_2$ is an unknown constant. There exist $\theta_2 \in \mathbb{R}$ and $l_2(x)\in C^{\infty}(\Omega)$ such that $u^{(I)}_2(x,t)$ can be written in:
\begin{equation}\label{eq:pee_SplitForm_uI}
    u^{(I)}_2(x,t)=e^{\theta_2 t}l_2(x).
\end{equation}

Multiplying both sides of \eqref{eq:pee_bar_u1} by $\omega$ and integrating by parts, we obtain
\begin{equation}\label{eq:pee_integral_r}
\int_{Q} (r_1-r_2)u^{(I)}_2(x,t) \omega(x,t) dxdt=0. 
\end{equation}

Substituting  \eqref{eq:pee_CGO_w} and \eqref{eq:pee_SplitForm_uI} into \eqref{eq:pee_integral_r} yields:
\begin{equation}\label{eq:pee_Substitute_r}
    \int^{T}_{0}e^{\theta_2 t}e^{(|\zeta|^2-r_1)t}dt \int_{\Omega} (r_1-r_2)l_2(x;\theta_2) e^{-\mathrm{i}\zeta \cdot x} dx=0, 
\end{equation}
which simplifies to:
\begin{equation}\label{eq:pee_VariableApart_r}
\int_{\Omega} (r_1-r_2)l_2(x;\theta_2) e^{-\mathrm{i}\zeta \cdot x} dx=0.\end{equation}
Since this holds for any Neumann eigenfunction $l_2(x;\theta_2)$ of $\Delta,$ we conclude that
\[r_1=r_2:=r.\]

At this stage, we can rewrite \eqref{eq:pee_bar_u1} as:
\begin{equation}\label{eq:pee_bar_u1_after}
    \begin{cases} 
\partial_t \bar{u}^{(I)}(x,t)-\Delta \bar{u}^{(I)}(x,t) -r\bar{u}^{(I)}(x,t)=0, &\  \text{in} \   Q,\\ 
\partial_{\nu}  \bar{u}^{(I)}(x,t)= \bar{u}^{(I)}(x,t)=0, &\  \text{on} \   \Sigma,\\ 
 \bar{u}^{(I)}(x,0)=\bar{u}^{(I)}(x,T)=0,& \  \text{in}\    \Omega. 
\end{cases} 
\end{equation}
This system has the trivial solution $\bar{u}^{(I)}(x,t)=0$. Due to the uniqueness of the solution under the specified boundary and initial conditions, it follows that $u^{(I)}_1 (x,t)=u^{(I)}_2(x,t):=u^{(I)}(x,t).$ Furthermore, the equation \eqref{eq:pee_SplitForm_uI} provides an alternative expression for $u^{(I)}(x,t)$.

\textbf{Recovery of source terms $\alpha^{10}(x)$ and $\alpha^{01}$.}
Next, we turn to the identification of $\alpha^{10}_j(x)$ and $\alpha^{01}_j$, which can be recovered simultaneously. Let $\bar{v}^{(I)}(x,t):=v^{(I)}_1 (x,t)-v^{(I)}_2(x,t)$. From \eqref{eq:pee_pf1st} and $\mathcal{M}^{+}_{A_1} = \mathcal{M}^{+}_{A_2}$, we obtain the following system for $\bar{v}^{(I)}(x,t)$:
\begin{equation}\label{eq:pee_bar_v1}
    \begin{cases} 
-\Delta \bar{v}^{(I)}(x,t)-\alpha^{01}_1\bar{v}^{(I)}(x,t) = (\alpha^{10}_1-\alpha^{10}_2)u^{(I)}(x,t)+(\alpha^{01}_1-\alpha^{01}_2)v^{(I)}_2(x,t), &\  \text{in} \   Q,\\ 
\partial_{\nu}  \bar{v}^{(I)}(x,t)= \bar{v}^{(I)}(x,t)=0, &\  \text{on} \   \Sigma,\\ 
 \bar{v}^{(I)}(x,0)=\bar{v}^{(I)}(x,T)=0,& \  \text{in}\    \Omega. 
\end{cases} 
\end{equation}

Let $\omega$ be a solution of the system 
\begin{equation}\label{eq:pee_omega2}
-\Delta \omega -\alpha^{01}_1 \omega=0 \  \text{in} \  \Omega,
\end{equation}
where $\alpha^{01}_1 $ is an unknown constant. 

First, choose $f_1(x)=0$.  Substituting into \eqref{eq:pee_pf1st}, we observe that $u^{(I)}(x,t)=0$ is a solution. By uniqueness of the heat equation, $u^{(I)}(x,t)$ must be trivial under this initial condition. Consequently, from \eqref{eq:pee_bar_v1}, we derive:
\begin{equation}\label{eq:pee_integral_beta1}
    \int_Q(\alpha^{01}_1-\alpha^{01}_2)v^{(I)}_2\omega dxdt=0.
\end{equation}

Without loss of generality, assume that $v^{(I)}_2(x,t)$ is independent of the spatial variable $x_n$ for $x=(x_1,x_2,\dots,x_n)\in \mathbb{R}^n$. Meanwhile, from \eqref{eq:pee_omega2}, a CGO solution is 
$\omega(x)=e^{\zeta\cdot x},x\in \mathbb{R}^n$, where $\vert \zeta\vert^2=-\alpha^{01}_1$, and $\zeta$ satisfies the following conditions:
\[\zeta=\xi+\mathrm{i}\xi^{\perp},\, \xi=(0,0,\dots,0,\xi_n)\in\mathbb{R}^n,\, \xi^{\perp}=(\xi^{\perp}_1,\dots,\xi^{\perp}_{n-1},0)\in\mathbb{R}^n,\]
with $\xi,\xi^{\perp}$ satisfying
\[ (\xi^{\perp}_1)^2+\cdots+(\xi^{\perp}_{n-1})^2=(\xi_n)^2 .\]

Separating the spatial variables in \eqref{eq:pee_integral_beta1}, we obtain:
\begin{align}\label{eq:pee_integral_beta1_repo}
   \int_Q&(\alpha^{01}_1-\alpha^{01}_2)v^{(I)}_2(x,t)\omega(x) dxdt = \\ \notag
   &\int_{\{x_n:(x^{\prime},x_n)\in \Omega\}}e^{\xi_n\cdot x_n}dx_n\cdot\int_{\{x^{\prime}:(x^{\prime},x_n)\in \Omega\}\times(0,T)}(\alpha^{01}_1-\alpha^{01}_2)v^{(I)}_2(x^{\prime},t)e^{\mathrm{i}\xi^{\prime}\cdot x^{\prime}}dx^{\prime}dt=0,
\end{align}
where $x^{\prime}=(x_1,\dots,x_{n-1})\in \mathbb{R}^{n-1},\xi^{\prime}=(\xi_1,\dots,\xi_{n-1})\in \mathbb{R}^{n-1}$. Since $\xi_{n}$ can be chosen arbitrarily, the term $\int_{\{x_n:(x^{\prime},x_n)\in \Omega\}}e^{\xi_n\cdot x_n}dx_n$ is non-zero. Therefore, \eqref{eq:pee_integral_beta1_repo} simplifies to:
\begin{equation}\label{eq:pee_integral_beta1_sep}
   \int_{\{x^{\prime}:(x^{\prime},x_n)\in \Omega\}\times(0,T)}(\alpha^{01}_1-\alpha^{01}_2)v^{(I)}_2(x^{\prime},t)e^{\mathrm{i}\xi^{\prime}\cdot x^{\prime}}dx^{\prime}dt=0, 
\end{equation}
which holds for any $\xi^{\prime}\in \mathbb{R}^{n-1}$. Let $A(x^{\prime})=(\alpha^{01}_1-\alpha^{01}_2)v^{(I)}_2(x^{\prime},t)$.  The left-hand side of \eqref{eq:pee_integral_beta1_sep} represents the Fourier transform of $\int^{T}_{0}A(x^{\prime},t)dt$.  By the inverse Fourier transform, $\int^{T}_{0}A(x^{\prime},t)dt=0$. Consequently, by the fundamental theorem of calculus, there exists a time $t_m$ such that $A(x^{\prime},t_m)=0$. 

Given any initial condition $g_1(x)>0$, we have $v^{(I)}_2(x^{\prime},t)>0$ by the maximum principle for elliptic equations, so $v^{(I)}_2(x^{\prime},t_m)>0$. Hence, to ensure $A(x^{\prime},t_m)=0$, we must have $\alpha^{01}_1=\alpha^{01}_2$. We denote this common value as  $\alpha^{01}$.

Once $\alpha^{01}$ is recovered, we can adjust the initial value $f_1$ to more general cases, which does not affect the derivation of the unknown coefficient function $\alpha^{10}_j(x)$. Consequently, from \eqref{eq:pee_bar_v1}, we obtain:
\begin{equation}\label{eq:pee_integral_alpha1}
    \int_Q(\alpha^{10}_1(x)-\alpha^{10}_2(x))u^{(I)}\omega dxdt=0.
\end{equation}

In the case of recovering $\alpha^{10}_j(x)$, we use a fundamental CGO solution $\omega$ to \eqref{eq:pee_omega2}, given by $e^{\mathrm{i}\zeta\cdot x}$, where $\vert \zeta\vert^2=-\alpha^{01}_1$. Meanwhile, by Lemma \ref{lem:pee_solForm}, $u^{(I)}(x,t)$ can be represented using \eqref{eq:pee_SplitForm_uI}. Thus, we transform \eqref{eq:pee_integral_alpha1} into:
\begin{equation}\label{eq:pee_VariableApart_alpha}
    \int^T_0 e^{\theta t}dt\int_{\Omega} (\alpha^{10}_1(x)-\alpha^{10}_2(x))l_2(x;\theta)e^{\mathrm{i}\zeta \cdot x}dx=0.
\end{equation}
Since $l_2(x;\theta)$ is any Neumann eigenfunction of $\Delta$, it follows that:
\begin{equation}\label{eq:pee_alpha1}
    \alpha^{10}_1(x)=\alpha^{10}_2(x).
\end{equation}
We denote this common function as $\alpha^{10}(x)$. Note that physically, $\alpha^{10}(x)$ and $\alpha^{01}$ represent different concepts: $\alpha^{10}(x)$ corresponds to the chemoattractant function, while $\alpha^{01}$ represents the decay rate.

At this stage, we can simplify \eqref{eq:pee_bar_v1} to:
\begin{equation}\label{eq:pee_bar_v1_after}
    \begin{cases} 
-\Delta \bar{v}^{(I)}(x,t) -\alpha^{01}_1\bar{v}^{(I)}(x,t)=0, &\  \text{in} \   Q,\\ 
\partial_{\nu}  \bar{v}^{(I)}(x,t)= \bar{v}^{(I)}(x,t)=0, &\  \text{on} \   \Sigma,\\ 
 \bar{v}^{(I)}(x,0)=\bar{v}^{(I)}(x,T)=0,& \  \text{in}\    \Omega. 
\end{cases} 
\end{equation}
This implies $ \bar{v}^{(I)}(x,t)=0$, so $v^{(I)}_1 (x,t)=v^{(I)}_2(x,t):=v^{(I)}(x,t)$.

\textbf{Recovery of the source terms $\beta^{10}(x)$ and $\beta^{01}$.}  We recover $\beta^{10}(x)$ and $\beta^{01}$ similarly. Let $\bar{w}^{(I)}(x,t):=w^{(I)}_1 (x,t)-w^{(I)}_2(x,t)$. From \eqref{eq:pee_pf1st} and $\mathcal{M}^{+}_{A_1}=\mathcal{M}^{+}_{A_2}$, we derive the following system for $\bar{w}^{(I)}(x,t)$:
\begin{equation}\label{eq:pee_bar_w1}
    \begin{cases} 
-\Delta \bar{w}^{(I)}(x,t)-\beta^{01}_1\bar{w}^{(I)}(x,t) = (\beta^{10}_1-\beta^{10}_2)u^{(I)}(x,t)+(\beta^{01}_1-\beta^{01}_2)w^{(I)}_2(x,t), &\  \text{in} \   Q,\\ 
\partial_{\nu}  \bar{w}^{(I)}(x,t)= \bar{w}^{(I)}(x,t)=0, &\  \text{on} \   \Sigma,\\ 
 \bar{w}^{(I)}(x,0)=\bar{w}^{(I)}(x,T)=0,& \  \text{in}\    \Omega. 
\end{cases} 
\end{equation}

Let $\omega$ be a solution of the following system 
\begin{equation}\label{eq:pee_omega3}
-\Delta \omega -\beta^{01}_1 \omega=0 \  \text{in} \  \Omega,
\end{equation}
where $\beta^{01}_1$ is an unknown constant. 

Next, we set $f_1(x)=0$, making $u^{(I)}(x,t)$ trivial. From \eqref{eq:pee_bar_w1}, we obtain:
\begin{equation}\label{eq:pee_integral_b01}
    \int_Q(\beta^{01}_1-\beta^{01}_2)w^{(I)}_2(x,t)\omega(x) dxdt=0,
\end{equation}

Similar to the recovery of $\alpha^{01}$, assume $w^{(I)}_2(x,t)$ is independent of the spatial variable $x_n$, where $x=(x_1,x_2,\dots,x_n)\in \mathbb{R}^n$. Choose the CGO solution for $\omega(x)$ to the equation \eqref{eq:pee_omega3} as
$\omega(x)=e^{\zeta\cdot x}$, $x\in \mathbb{R}^n$, where $\vert \zeta\vert^2=-\beta^{01}_1$, and $\zeta$ satisfies the following conditions:
\[\zeta=\xi+\mathrm{i}\xi^{\perp},\, \xi=(0,0,\dots,0,\xi_n)\in\mathbb{R}^n,\, \xi^{\perp}=(\xi^{\perp}_1,\dots,\xi^{\perp}_{n-1},0)\in\mathbb{R}^n,\]
with $\xi,\xi^{\perp}$ satisfying
\[ (\xi^{\perp}_1)^2+\cdots+(\xi^{\perp}_{n-1})^2=(\xi_n)^2 .\]

Then \eqref{eq:pee_integral_b01} becomes:
\begin{align}\label{eq:pee_integral_b01_repo}
   \int_Q&(\beta^{01}_1-\beta^{01}_2)w^{(I)}_2(x,t)\omega(x) dxdt = \\ \notag
   &\int_{\{x_n:(x^{\prime},x_n)\in \Omega\}}e^{\xi_n x_n}dx_n\cdot\int_{\{x^{\prime}:(x^{\prime},x_n)\in \Omega\}\times(0,T)}(\beta^{01}_1-\beta^{01}_2)w^{(I)}_2(x^{\prime},t)e^{\mathrm{i}\xi^{\prime}\cdot x^{\prime}}dx^{\prime}dt=0,
\end{align}
where $x^{\prime}=(x_1,\dots,x_{n-1})\in \mathbb{R}^{n-1},\xi^{\prime}=(\xi_1,\dots,\xi_{n-1})\in \mathbb{R}^{n-1}$. Since $\xi_{n}$ is arbitrarily chosen, the term $\int_{\{x_n:(x^{\prime},x_n)\in \Omega\}}e^{\xi_n x_n}dx_n$ is non-zero. Thus, \eqref{eq:pee_integral_b01_repo} simplifies to:
\begin{equation}\label{eq:pee_integral_b01_sep}
   \int_{\{x^{\prime}:(x^{\prime},x_n)\in \Omega\}\times(0,T)}(\beta^{01}_1-\beta^{01}_2)w^{(I)}_2(x^{\prime},t)e^{\mathrm{i}\xi^{\prime}\cdot x^{\prime}}dx^{\prime}dt=0, 
\end{equation}
which holds for any $\xi^{\prime}\in \mathbb{R}^{n-1}$. Let $B(x^{\prime})=(\beta^{01}_1-\beta^{01}_2)w^{(I)}_2(x^{\prime},t)$.  The left-hand side of \eqref{eq:pee_integral_b01_sep} represents the Fourier transform of $\int^{T}_{0}B(x^{\prime},t)dt$.  By the inverse Fourier transform, $\int^{T}_{0}B(x^{\prime},t)dt=0$. Consequently, by the Fundamental Theorem of Calculus, there exists a time $t_m$ such that $B(x^{\prime},t_m)=0$. 

Given any initial condition $h_1(x)>0$, we have $w^{(I)}_2(x^{\prime},t)>0$ by the maximum principle for elliptic equations, which implies $w^{(I)}_2(x^{\prime},t_m)>0$. Hence, to ensure $B(x^{\prime},t_m)=0$, we conclude that $\beta^{01}_1=\beta^{01}_2$, and we unify its notation as $\beta^{01}$.

Next, to recover $\beta^{10}(x)$, we reset the initial condition to $f_1>0$. Then \eqref{eq:pee_bar_w1} gives:
\begin{equation}\label{eq:pee_integral_gamma1}
    \int_Q(\beta^{10}_1(x)-\beta^{10}_2(x))u^{(I)}(x,t)\omega(x) dxdt=0.
\end{equation}
It is known that $u^{(I)}(x,t)$ can be represented using \eqref{eq:pee_SplitForm_uI}, and we give the CGO solution $\omega$ for \eqref{eq:pee_omega3} as $\omega(x)=e^{\mathrm{i}\zeta\cdot x}$, with $\vert \zeta \vert^2=-\beta^{01}_1$. Therefore, we transform \eqref{eq:pee_integral_gamma1} into:
\begin{equation}\label{eq:pee_VariableApart_beta10}
    \int^T_0 e^{\theta_2t}dt\int_{\Omega} (\beta^{10}_1(x)-\beta^{10}_2(x))l_2(x;\theta)e^{\mathrm{i}\zeta \cdot x}dx=0.
\end{equation}
Since $l_2(x;\theta)$ is any Neumann eigenfunction of $\Delta$, we obtain the following result:
\begin{equation}\label{eq:pee_beta10}
    \beta^{10}_1(x)=\beta^{10}_2(x).
\end{equation}
Denote this as $\beta^{10}(x)$. 

Substituting these results into \eqref{eq:pee_bar_w1} yields:
\begin{equation}\label{eq:pee_bar_w1_after}
    \begin{cases} 
-\Delta \bar{w}^{(I)}(x,t)-\beta^{01}_1\bar{w}^{(I)}(x,t)=0, &\  \text{in} \   Q,\\ 
\partial_{\nu}  \bar{w}^{(I)}(x,t)= \bar{w}^{(I)}(x,t)=0, &\  \text{on} \   \Sigma,\\ 
 \bar{w}^{(I)}(x,0)=\bar{w}^{(I)}(x,T)=0,& \  \text{in}\    \Omega, 
\end{cases} 
\end{equation}
which implies $ \bar{w}^{(I)}(x,t)=0$, leading to $w^{(I)}_1 (x,t)=w^{(I)}_2(x,t):=w^{(I)}(x,t)$.

\subsection{Recovery of the second-order coefficients} \label{sec:pee_mainR2}
In this subsection, we begin by introducing the second-order variation form associated with the system \eqref{eq:pee_mainpfuse}. Similar to the definition in \ref{sec:pee_mainR1},
we consider 
\[u^{(II)}_j:=\partial_{\varepsilon}^2u_j|_{\varepsilon=0},\, 	v^{(II)}_j:=\partial_{\varepsilon}^2v_j|_{\varepsilon=0},\, \text{and}\, w^{(II)}_j:=\partial_{\varepsilon}^2w_j|_{\varepsilon=0}.\]
$(u^{(II)}_j,v^{(II)}_j,w^{(II)}_j)$ can be interpreted as the output of the second-order Fr\'echet derivatives of $S$ at a specific point. Then, we have the second-order variation as follows:


\begin{equation}\label{eq:pee_pf2nd}
    \begin{cases} 
\partial_t u^{(II)}_{j}=\Delta u^{(II)}_j+r u^{(II)}_j-2\chi_j\nabla u^{(I)}\nabla v^{(I)}-2\chi_ju^{(I)}\Delta v^{(I)}\\
\qquad \qquad \qquad+2\xi_j\nabla u^{(I)}\nabla w^{(I)}+2\xi_ju^{(I)}\Delta w^{(I)}-2\mu_j(u^{(I)})^2, &\  \text{in} \    Q,\\ 
0=\Delta v^{(II)}_j +\alpha^{10} u^{(II)}_j-\alpha^{01} v^{(II)}_j+\alpha^{11}_j u^{(I)}v^{(I)}+2\alpha^{20}_j(u^{(I)})^2+2\alpha^{02}_j(v^{(I)})^2, &\  \text{in} \    Q,\\ 
0=\Delta w^{(II)}_j  +\beta^{10} u^{(II)}_j-\beta^{01} w^{(II)}_j+\beta^{11}_j u^{(I)}w^{(I)}+2\beta^{20}_j(u^{(I)})^2+2\beta^{02}_j(w^{(I)})^2, &\  \text{in} \    Q,\\ 
\partial_{\nu}u^{(II)}_{j}=\partial_{\nu}v^{(II)}_{j}=\partial_{\nu}w^{(II)}_{j}=0 ,&\  \text{on}\    \Sigma,\\ 
u^{(II)}_j (x,0)=2f_2(x),\, v^{(II)}_j(x,0)=2g_2(x),\, w^{(II)}_j(x,0)=2h_2(x), &\  \text{in} \   \Omega.\\ 
\end{cases} 
\end{equation}

Note that the non-linear terms of the system \eqref{eq:pee_pf2nd} depend on the first-order linearized system \eqref{eq:pee_pf1st}, all the conclusions we obtained from \eqref{eq:pee_pf1st} also apply to \eqref{eq:pee_pf2nd}. 

\textbf{Recovery of chemosensitivity $\chi$, $\xi$ and self-suppression coefficient $\mu$.} 
From \eqref{eq:pee_pf1st}, it is clear that $\Delta v^{(I)}$ and $\Delta w^{(I)}$ satisfy:
\[\Delta v^{(I)}(x,t)=-\alpha^{10}(x) u^{(I)}(x.t)-\alpha^{01} v^{(I)}(x,t),\]
\[ \Delta w^{(I)}(x,t)=-\beta^{10}(x) u^{(I)}(x,t)-\beta^{01} w^{(I)}(x,t).\]
By controlling the initial data so that $\Delta v^{(I)}(x,t)=\Delta w^{(I)}(x,t)=0$, we substitute the relationships between $v^{(I)}(x,t),w^{(I)}(x,t)$ and $u^{(I)}(x,t)$ into the first equation of \eqref{eq:pee_pf2nd}. This leads to the transformation:
\begin{align}\label{eq:pee_uSec_transformed}
    \partial_t u^{(II)}_{j}&=\Delta u^{(II)}_j+r u^{(II)}_j+\frac{2\chi_j}{\alpha^{01}}\left[\nabla \alpha^{10}u^{(I)}\nabla u^{(I)}+\alpha^{10}(\nabla u^{(I)})^2\right] \notag\\
    &\qquad\qquad\qquad\qquad\qquad\quad -\frac{2\xi_j}{\beta^{01}}\left[\nabla \beta^{10}u^{(I)}\nabla u^{(I)}+\beta^{10}(\nabla u^{(I)})^2\right]
    -2\mu_j(u^{(I)})^2.
\end{align}
For simplicity, denote $C(x,t)=\nabla \alpha^{10}(x)u^{(I)}(x,t)\nabla u^{(I)}(x,t)+\alpha^{10}(x)(\nabla u^{(I)}(x,t))^2$ and $D(x,t)=\nabla \beta^{10}(x)u^{(I)}(x,t)\nabla u^{(I)}(x,t)+\beta^{10}(x)(\nabla u^{(I)}(x,t))^2$.

Let $\bar{u}^{(II)}(x,t):=u^{(II)}_{1}(x,t)-u^{(II)}_{2}(x,t)$. From \eqref{eq:pee_pf2nd}, \eqref{eq:pee_uSec_transformed}, and $\mathcal{M}^{+}_{A_1} = \mathcal{M}^{+}_{A_2}$, we obtain:
\begin{equation}\label{eq:pee_bar_u2}
    \begin{cases} 
\partial_t \bar{u}^{(II)}(x,t)-\Delta \bar{u}^{(II)}(x,t) = r \bar{u}^{(II)}(x,t)+\frac{2}{\alpha^{01}}(\chi_1-\chi_2)C(x,t)\\
\qquad\qquad\qquad\qquad\qquad\qquad\ -\frac{2}{\beta^{01}}(\xi_1-\xi_2)D(x,t)-2(\mu_1-\mu_2)(u^{(I)}(x,t))^2, &\  \text{in} \   Q,\\ 
\partial_{\nu}  \bar{u}^{(II)}(x,t)= \bar{u}^{(II)}(x,t)=0, &\  \text{on} \   \Sigma,\\ 
 \bar{u}^{(II)}(x,0)=\bar{u}^{(II)}(x,T)=0,& \  \text{in}\    \Omega. 
\end{cases} 
\end{equation}

Let $\omega$ be a solution of  
\begin{equation}\label{eq:pee_omega4}
-\partial_{t }\omega -\Delta \omega -r \omega=0 \  \text{in} \  Q,
\end{equation}
where $r$ is an unknown constant.

Multiplying both sides of \eqref{eq:pee_bar_u2} by $\omega$ and integrating by parts, we achieve
\begin{equation}\label{eq:pee_three_coefficients}
    \int_Q \left[\frac{2}{\alpha^{01}}(\chi_1-\chi_2)C-\frac{2}{\beta^{01}}(\xi_1-\xi_2)D-2(\mu_1-\mu_2)(u^{(I)})^2\right]\omega dxdt=0.
\end{equation}

The recovery of the coefficients can then be divided into three cases.

First, assume $\xi_1 = \xi_2$ and $\mu_1 = \mu_2$. Then \eqref{eq:pee_three_coefficients} reduces to:
\begin{equation}\label{eq:pee_chi}
    \int_Q \frac{2}{\alpha^{01}}(\chi_1-\chi_2)C(x,t) \omega(x,t) dxdt=0.
\end{equation}
Choose the CGO solution for $\omega(x,t)$ to \eqref{eq:pee_omega4} as $\omega(x,t)=e^{(-\vert\zeta\vert^2-r)t+\zeta\cdot x},$  where $\zeta\in\mathbb{R}$ and satisfies:
\[\zeta=\eta+\mathrm{i}\eta^{\perp},\, \eta=(0,0,\dots,0,\eta_n)\in\mathbb{R}^n,\, \eta^{\perp}=(\eta^{\perp}_1,\dots,\eta^{\perp}_{n-1},0)\in\mathbb{R}^n,\]
with $\eta,\eta^{\perp}$ satisfying $(\eta^{\perp}_1)^2+\cdots+(\eta^{\perp}_{n-1})^2=(\eta_n)^2.$
Then \eqref{eq:pee_chi} can be written as:
\begin{equation}\label{eq:pee_chi_repo}\begin{small}
   \int_{\{x_n:(x^{\prime},x_n)\in \Omega\}}e^{\eta_n x_n}dx_n\cdot\int_{\{x^{\prime}:(x^{\prime},x_n)\in \Omega\}\times(0,T)}\frac{2(\chi_1-\chi_2)}{\alpha^{01}}C(x^{\prime},t)e^{(-|\zeta|^2-r)t+\mathrm{i}\eta^{\prime}\cdot x^{\prime}}dx^{\prime}dt=0,\end{small}
\end{equation}
where $x^{\prime}=(x_1,\dots,x_{n-1})\in \mathbb{R}^{n-1},\eta^{\prime}=(\eta_1,\dots,\eta_{n-1})\in \mathbb{R}^{n-1}$. Since $\eta_{n}$ is chosen arbitrarily, the term $\int_{\{x_n:(x^{\prime},x_n)\in \Omega\}}e^{\eta_n x_n}dx_n$ is non-zero. Therefore, \eqref{eq:pee_chi_repo} simplifies to:
\begin{equation}\label{eq:pee_integral_chi_sep}
   \int_{\{x^{\prime}:(x^{\prime},x_n)\in \Omega\}\times(0,T)}\frac{2(\chi_1-\chi_2)}{\alpha^{01}}C(x^{\prime},t)e^{(-|\zeta|^2-r)t+\mathrm{i}\eta^{\prime}\cdot x^{\prime}}dx^{\prime}dt=0, 
\end{equation}
which holds for any $\eta^{\prime}\in \mathbb{R}^{n-1}$. It is clear that the left-hand side of \eqref{eq:pee_integral_chi_sep} represents the Fourier transform of $\int^{T}_{0}\frac{2(\chi_1-\chi_2)}{\alpha^{01}}C(x^{\prime},t)e^{(-|\zeta|^2-r)t}dt$.  By the inverse Fourier transform, $\int^{T}_{0}\frac{2(\chi_1-\chi_2)}{\alpha^{01}}C(x^{\prime},t)e^{(-|\zeta|^2-r)t}dt=0$. Consequently, by the Fundamental Theorem of Calculus, there exists a time $t_m$ such that 
\begin{equation}\label{eq:pee_chi_conclud}
    \frac{2(\chi_1-\chi_2)}{\alpha^{01}}C(x^{\prime},t_m)e^{(-|\zeta|^2-r)t_m}=0.
\end{equation}

Now we expand $C(x^{\prime},t_m)$ to $\nabla \alpha^{10}(x)u^{(I)}(x,t_m)\nabla u^{(I)}(x,t_m)+\alpha^{10}(x)(\nabla u^{(I)}(x,t_m))^2.$ By Lemma \ref{lem:pee_solForm}, $u^{(I)}(x,t_m)$ can be expressed in $e^{\theta t_m}l(x;\theta)$, and there does not exist an open subset $U$ of $\Omega$ such that $\nabla l(x;\theta)=0$. Equation \eqref{eq:pee_chi_conclud} now indicates
\begin{equation}\label{eq:pee_chi_conclud_repo}
    \frac{2(\chi_1-\chi_2)}{\alpha^{01}}\left[\nabla \alpha^{10}(x^{\prime})e^{2\theta t_m}l(x^{\prime};\theta)\nabla l(x^{\prime};\theta)+\alpha^{10}(x^{\prime}) e^{2\theta t_m} (\nabla l(x^{\prime};\theta))^2 \right]e^{(-|\zeta|^2-r)t_m}=0.
\end{equation}

\sloppy 
Suppose $C(x^{\prime},t_m)=0$. By the definition of $C$, this requires $\nabla \alpha^{10}(x^{\prime})l(x^{\prime};\theta)=\alpha^{10}(x^{\prime})\nabla l(x^{\prime}; \theta)$, which implies $\alpha^{10}(x^{\prime})=c l(x^{\prime};\theta)$, for some constant $c$. However, $l(x^{\prime};\theta)$ is any eigenfunction of the equations for $u^{(I)}(x^{\prime},t)$ and satisfies
\begin{equation}\notag\begin{cases}
    \Delta l(x;\theta)+(r-\theta)l(x;\theta)=0,&\  \text{in} \    Q,\\ 
    \partial_{\nu} l(x;\theta)=0,&\  \text{on} \    \Sigma,\\
    l(x;\theta)=f_1(x),&\  \text{in} \    \Omega.
 \end{cases}   
\end{equation}
Thus, there such an $\alpha^{10}(x^{\prime})$ cannot exist, meaning that the assumption cannot hold. Therefore, for any initial condition $f_1(x)>0$, $C(x^{\prime},t_m)$ must be non-zero. To satisfy \eqref{eq:pee_chi_conclud_repo}, we must have $\chi_1=\chi_2$, denoted as $\chi$.


Similarly, if $\chi_1=\chi_2$ and $\mu_1=\mu_2$ are known, we can recover $\xi$ in a same way.

For the third case, assume $\chi$ and $\xi$ are known. Then, \eqref{eq:pee_three_coefficients} gives:
\begin{equation}\label{eq:pee_mu}
     \int_Q 2(\mu_1-\mu_2)(u^{(I)})^2\omega dxdt=0.
\end{equation}
Here, we choose a simpler form of CGO solution $\omega$ to \eqref{eq:pee_omega4} as 
\begin{equation}\label{eq:pee_CGO_w4}
    \omega=e^{(\vert \zeta\vert^2-r)t-\mathrm{i}\zeta\cdot x}, \text{ with } \mathrm{i}=\sqrt{-1}\text{ for }\zeta \in \mathbb{R}^n.
\end{equation}

Substituting \eqref{eq:pee_SplitForm_uI} and \eqref{eq:pee_CGO_w4} into \eqref{eq:pee_mu} and separating variables, we obtain:
\begin{equation}\label{eq:pee_mu_Sep}
   2 \int_0^T e^{(2\mu_2+|\zeta|^2-r)t}dt \int_{\Omega}(\mu_1-\mu_2)l^2(x) e^{-\mathrm{i}\zeta\cdot x}dx=0.
\end{equation}
Since this holds for any Neumann eigenfunction $l(x;\mu)$ of $\Delta$, we conclude $\mu_1=\mu_2$, and denoted as $\mu$.


Substituting these results into \eqref{eq:pee_bar_u2} gives:
\begin{equation}\label{eq:pee_bar_u2_after}
    \begin{cases}
        \partial_t \bar{u}^{(II)}(x,t)-\Delta \bar{u}^{(II)}(x,t) = r \bar{u}^{(II)}(x,t), &\  \text{in} \   Q,\\ 
\partial_{\nu}  \bar{u}^{(II)}(x,t)= \bar{u}^{(II)}(x,t)=0, &\  \text{on} \   \Sigma,\\ 
 \bar{u}^{(II)}(x,0)=\bar{u}^{(II)}(x,T)=0,& \  \text{in}\    \Omega. 
    \end{cases}
\end{equation}
It is evident that $\bar{u}^{(II)}(x,t)=0$ is a solution to \eqref{eq:pee_bar_u2_after}. Given the uniqueness of the solution under the specified boundary and initial conditions, we conclude $u^{(II)}_1(x,t)=u^{(II)}_2(x,t):=u^{(II)}(x,t)$.

\textbf{Recovery of the second-order coefficients of source term $\alpha^{11}(x),\alpha^{20}(x)$ and $\alpha^{02}(x)$.} 
Let $\bar{v}^{(II)}(x,t):=\bar{v}^{(II)}_1(x,t)-\bar{v}^{(II)}_2(x,t)$. From \eqref{eq:pee_pf2nd} and $\mathcal{M}^{+}_{A_1} = \mathcal{M}^{+}_{A_2}$, we obtain:
\begin{equation}\label{eq:pee_bar_v2}
    \begin{cases}
       -\Delta \bar{v}^{(II)}(x,t)-\alpha^{01}\bar{v}^{(II)}(x,t) \\\qquad= \left(\alpha^{11}_1(x)-\alpha^{11}_2(x)\right)u^{(I)}(x,t)v^{(I)}(x,t)\\\qquad\quad+2\left(\alpha^{20}_1(x)-\alpha^{20}_2(x)\right)\left(u^{(I)}(x,t)\right)^2+2\left(\alpha^{02}_1(x)-\alpha^{02}_2(x)\right)\left(v^{(I)}(x,t)\right)^2, &\  \text{in} \   Q,\\ 
\partial_{\nu}  \bar{v}^{(II)}(x,t)= \bar{v}^{(II)}(x,t)=0, &\  \text{on} \   \Sigma,\\ 
 \bar{v}^{(II)}(x,0)=\bar{v}^{(II)}(x,T)=0,& \  \text{in}\    \Omega.
    \end{cases}
\end{equation}
First, choose $f_1(x)=0$ and $g_1(x)>0$. Following a similar reasoning as in the previous proof, it follows that $u^{(I)}(x,t)$ must be trivial. Consequently, by taking the solution $\omega(x)$ of \eqref{eq:pee_omega2},  from \eqref{eq:pee_bar_v2}, we derive:
\begin{equation}\label{eq:pee_integral_alpha02}
    \int_Q2(\alpha^{02}_1(x)-\alpha^{02}_2(x))(v^{(I)}(x,t))^2\omega(x) dxdt=0.
\end{equation}

Here we choose the CGO solution for $\omega(x)$ to \eqref{eq:pee_omega2} as $\omega(x)=e^{\zeta\cdot x},$  where $\zeta\in\mathbb{R}$ is of the form in Lemma \ref{lem:pee_MultiSepe}. Then by Lemma \ref{lem:pee_MultiSepe} and the fundamental theorem of calculus, we have 
\[(\alpha^{02}_1(x)-\alpha^{02}_2(x))(v^{(I)}(x,t_m))^2=0,\quad t_m\in(0,T).\]
Since $g_1(x)>0$, by the maximum principle, $v^{(I)}>0$. Therefore, it follows that $\alpha^{02}_1(x)=\alpha^{02}_2(x)$ in $\Omega$.

Next, choose $g_1(x)=0$ and $f_1(x)>0$ to recover $\alpha^{20}(x)$ similarly. 


Finally, choose $f_1(x),g_1(x)>0$, which leads us to $u^{(I)},v^{(I)}>0$, and similarly apply Lemma \ref{lem:pee_MultiSepe} to recover $\alpha^{11}(x)$. 

Then the equations for $\bar{v}^{(II)}(x,t)$ become
\begin{equation}\notag
    \begin{cases}
       -\Delta \bar{v}^{(II)}(x,t)-\alpha^{01}\bar{v}^{(II)}(x,t) = 0, &\  \text{in} \   Q,\\ 
\partial_{\nu}  \bar{v}^{(II)}(x,t)= \bar{v}^{(II)}(x,t)=0, &\  \text{on} \   \Sigma,\\ 
 \bar{v}^{(II)}(x,0)=\bar{v}^{(II)}(x,T)=0,& \  \text{in}\    \Omega,
    \end{cases}
\end{equation}
which implies $v^{(II)}_1(x,t)=v^{(II)}_2(x,t)$, denoted as $v^{(II)}(x,t)$.

\textbf{Recovery of the second-order coefficients of the source term $\beta^{11}(x),\beta^{20}(x)$ and $\beta^{02}(x)$.} 
To recover these three coefficients, we follow a similar approach as we recover $\alpha^{11}(x),\alpha^{20}(x)$ and $\alpha^{02}(x)$.

\subsection{Recovery of the higher-order coefficients of the source term} \label{sec:pee_mainR+}
In this subsection, we introduce the high-order variation framework under a more general setting. 
Inductively, for $\ell\in\mathbb{N}, N>2$, we define
\[u^{(\ell)}_j=\partial_{\varepsilon}^\ell u_j|_{\varepsilon=0}, \,	v^{(\ell)}_j=\partial_{\varepsilon}^\ell v_j|_{\varepsilon=0}, \,\text{and} \, w^{(\ell)}_j=\partial_{\varepsilon}^\ell w_j|_{\varepsilon=0},\]
and obtain a sequence of parabolic-elliptic-elliptic systems for $\tau=0$.

The main idea for recovering higher-order coefficients $\alpha^{k_1 k_2}$ and $\beta^{k_1 k_2}$ with $k_1 + k_2 = k \geq 3$ is mathematical induction, based on the $k$-th variation of \eqref{eq:pee_mainpfuse}.

Thus, the proof is complete.
$\hfill{\square}$

{\centering \section{APPLICATIONS}   \label{sec:pee_application}}
In this section, we apply the above conclusion to an attraction-repulsion chemotaxis system with superlinear logistic degradation. We demonstrate how to recover the coefficients simultaneously under $\tau=1$ and $\tau=0$.

We recall that the system describes the spatiotemporal dynamics of a biological population $u$ (e.g., cells, bacteria) interacting with two chemical signals 
$v$ (attractant) and $w$ (repellent). The well-posedness of \eqref{eq:pee_apply_main} is discussed in subsection \ref{sec:pee_mainrs}. When $\tau=1$, the system is classified as a total parabolic system, while it becomes a parabolic-elliptic-elliptic system when $\tau=0$. Given the measurement map in \eqref{eq:pee_Measurement2}, we aim to prove Corollary \ref{cor:pee_appcor} from both perspectives: $\tau=1$ and $\tau=0$.

\begin{proof} We carry out the proof in two steps, $\tau=0$ and $\tau=1$. 

For the case $\tau=0$, it is evident that the system \eqref{eq:pee_apply_prop} is a simpler form of the system \eqref{eq:pee_mainpfuse} while $\alpha_j=\alpha^{10}_j$, $\beta_j=\alpha^{01}_j$, $\gamma_j=\beta^{10}_j$, $\delta_j=\beta^{01}_j$ and all the other $\alpha^{pq}_j=\beta^{rs}_j=0$ for $p+q>2$ and $r+s>2$. Thus, we can conclude that $B_1=B_2$ using the same method as in the main proof.

For the case $\tau=1$, the equations \eqref{eq:pee_apply_prop} is a parabolic system, the recovery of the coefficients $\chi_j, \xi_j,r_j$ and $\mu_j$ are identical to the above proof. And we recover the coefficient functions $\alpha_j,\beta_j,\gamma_j$ and $\delta_j$ in a way as in parabolic systems. Without losing generality, we recover $\alpha_j$ and $\beta_j$, the recovery of $\gamma_j$ and $\delta_j$ follow an identical approach.

Following the same procedure for constructing high-order variation forms as outlined in Section \ref{sec:pee_pfmain}, we derive the first-order variation system for $u(x,t)$ and $v(x,t)$ of \eqref{eq:pee_apply_prop} as:
\begin{equation}\label{eq:pee_apply_v}
    \begin{cases} 
\partial_t u^{(I)}_{j}(x,t)=\Delta u^{(I)}_j(x,t)+r_j u^{(I)}_j(x,t), &\  \text{in} \    Q,\\ 
\partial_t v^{(I)}_{j}(x,t)=\Delta v^{(I)}_j (x,t)+\alpha_ju^{(I)}_j(x,t)-\beta_jv^{(I)}_j(x,t),\, &\  \text{in} \    Q,\\  
\partial_{\nu}u^{(I)}_{j}(x,t)=\partial_{\nu}v^{(I)}_{j}(x,t)=0 ,&\  \text{on}\    \Sigma,\\ 
u^{(I)}_j (x,0)=f_1(x),\, v^{(I)}_j(x,0)=g_1(x), &\  \text{in} \   \Omega.\\ 
\end{cases} 
\end{equation}

Following the same approach in Section \ref{sec:pee_pfmain}, we know that $ u^{(I)}_{1}(x,t)= u^{(I)}_{2}(x,t)$. We can denote it as $u^{(I)}(x,t)$, and can express it as $e^{\mu t}l(x;\mu)$ based on Lemma \ref{lem:pee_solForm}. Let $\bar{v}^{(I)}(x,t)=v^{(I)}_1 (x,t)-v^{(I)}_2 (x,t)$, we can obtain the following system based on $\mathcal{M}^{+}_{B_1}(u_0,v_0,w_0)=\mathcal{M}^{+}_{B_2}(u_0,v_0,w_0)$:
\begin{equation}\label{eq:pee_apply_bar_v1}
    \begin{cases} 
\partial_t\bar{v}^{(I)}(x,t)-\Delta \bar{v}^{(I)}(x,t)+\beta_2\bar{v}^{(I)}(x,t) = (\alpha_1-\alpha_2)u^{(I)}(x,t)-(\beta_1-\beta_2)v^{(I)}_1(x,t), &\  \text{in} \   Q,\\ 
\partial_{\nu}  \bar{v}^{(I)}(x,t)= \bar{v}^{(I)}(x,t)=0, &\  \text{on} \   \Sigma,\\ 
 \bar{v}^{(I)}(x,0)=\bar{v}^{(I)}(x,T)=0,& \  \text{in}\    \Omega. 
\end{cases} 
\end{equation}

Let $\omega$ be a solution of the following system 
\begin{equation}\label{eq:pee_apply_omega}
-\partial_{t }\omega -\Delta \omega +\beta_2 \omega=0 \  \text{in} \  Q,
\end{equation}
where $\beta_2$ is an unknown constant, and the CGO solution to $\omega$ is easy to seek from \eqref{eq:pee_apply_omega} as
\begin{equation}\label{eq:pee_apply_CGO}
\omega=e^{(|\xi|^2+\beta_2)t-\mathrm{i}\xi \cdot x},
\end{equation}
with $\mathrm{i}=\sqrt{-1}$ for $\xi \in \mathbb{R}^n.$

Then we multiply $\omega$ on both sides of \eqref{eq:pee_apply_bar_v1} and carry on integration by parts, we now achieve
\begin{equation}\label{eq:pee_apply_integration}
\int_{Q} \big[(\alpha_1-\alpha_2)u^{(I)}(x,t)-(\beta_1-\beta_2)v^{(I)}_1(x,t)\big] \omega dxdt=0. \end{equation}

To recover $\beta_j$, we can choose $f_1(x)=0$, so equations \eqref{eq:pee_apply_v} indicate $u^{(I)}(x,t)=0$. And $v^{(I)}_j(x,t)$ in \eqref{eq:pee_apply_v} satisfies the form asked in Lemma \ref{lem:pee_solForm} and can be expressed in the form of $v^{(I)}(x,t)=e^{\lambda t}m(x;\lambda)$, where $\lambda\in\mathbb{R}^{n}$ and $m(x;\lambda)\in C^2(\Omega)$. And the equation \eqref{eq:pee_apply_integration} can transform into:
\begin{equation}\notag
    \int^{T}_{0}e^{\lambda t}e^{(|\xi|^2+\beta_2)t}dt \int_{\Omega} (\beta_1-\beta_2)m(x;\lambda) e^{-\mathrm{i}\xi \cdot x} dx=0, 
\notag\end{equation}
which yields 
\begin{equation}\notag
\int_{\Omega} (\beta_1-\beta_2)m(x;\lambda) e^{-\mathrm{i}\xi \cdot x}  dx=0.
\notag\end{equation}
Since this holds for any Neumann eigenfunction $m(x;\lambda)$ of $\Delta,$ we obtain
\[\beta_1=\beta_2=:\beta.\]

Then equation \eqref{eq:pee_apply_integration} only contains the term including $u^{(I)}(x,t)$. Once again, we substitute the CGO form of $\omega$ on both sides of \eqref{eq:pee_apply_integration} and separate the variables:
\begin{equation}\notag
    \int^{T}_{0}e^{\mu t}e^{(|\xi|^2+\beta_2)t}dt \int_{\Omega} (\alpha_1-\alpha_2)l(x;\lambda) e^{-\mathrm{i}\xi \cdot x} dx=0, 
\notag\end{equation}
which yields
\begin{equation}\notag
\int_{\Omega} (\alpha_1-\alpha_2)l(x;\lambda) e^{-\mathrm{i}\xi \cdot x}  dx=0.
\notag\end{equation}
It is known that this holds for any Neumann eigenfunction $l(x;\mu)$ of $\Delta,$ we obtain
\[\alpha_1=\alpha_2=:\alpha.\]
Also note that $\alpha$ can be a function depending on the space variable $x$, which does not affect the proving process and simultaneously broadens the field of application. The same reasoning applies to the recovery of $\gamma$ and $\delta$.

In this way, we can recover all the coefficient functions in the set $B$, and hence finish the proof for Proposition \ref{cor:pee_appcor}. \end{proof}

\section*{Acknowledgements}
The work of H. Liu is partially supported by the Hong Kong RGC General Research Funds (projects 11303125, 11304224 and
11311122). The work of C. W. K. Lo is supported by the National Natural Science Foundation of China (No. 12501660), the Guangdong Province ``Pearl River" Talent Recruitment Program (Top youth talent) (No. 2024QN11X268), the Shenzhen Science and Technology Program (No. QNXMC20250701094200001) and the Guangdong Key Laboratory of Applied Mathematics and Artificial Intelligence.

\vspace{1em}
\noindent\textbf{Data availability statement.} 
We do not analyse or generate any datasets, because our work proceeds within a theoretical and mathematical approach.

\bibliographystyle{plain}
\bibliography{reference}

\end{document}